\documentclass[letterpaper, 11pt]{article}
\usepackage[margin=0.75in]{geometry}

\usepackage[]{hyperref} %for href
\usepackage{enumitem}
\usepackage{blindtext}
\usepackage{graphicx}
\usepackage{amssymb,amsmath,natbib}
\usepackage{url}
\usepackage{comment}
\usepackage{bm}
\usepackage{subfig}
\usepackage{amsthm,mathtools} %for a proof environment
\usepackage{bbm} %for unit tensor $\mathbbm{1}$

\usepackage[font=small]{caption}
\usepackage[mathscr]{euscript} % this is for \mathscr definition of tensors
\usepackage{xcolor}

\newcommand{\ten}[1]{\boldsymbol{\mathscr{#1}}}

\newcommand{\mat}[1]{\mathbf{#1}}

\newcommand{\vct}[1]{\boldsymbol{#1}}
\newcommand{\bi}{\vct{i}}
\newcommand{\bx}{\vct{x}} 
\newcommand{\bz}{\vct{z}} 

\newcommand{\ba}{\boldsymbol a}

\newcommand{\by}{\boldsymbol y}

\newcommand{\bs}{\boldsymbol s}

\newcommand{\bone}{\boldsymbol 1} 
\newcommand{\bzero}{\boldsymbol 0} 

\newcommand{\blambda}{\boldsymbol \lambda} 
 
\newcommand{\bmu}{\boldsymbol \mu} 
 
\newcommand{\btheta}{\vct{\theta}} 

\newtheorem{theorem}{Theorem}
\newtheorem{lemma}{Lemma}
\newtheorem{corollary}{Corollary}
\newtheorem{conjecture}{Conjecture}

\newcommand{\vecc}{\text{vec} }
\newcommand{\dv}{\text{dvec} }
\newcommand{\diag}{\text{diag} }

\newcommand{\mbbE}{\mathbb E} 
\newcommand{\mbbR}{\mathbb R} 

\newcommand{\Poi}{\text{Poisson} }

\newcommand{\FIM}{\mathcal I }

\DeclareMathOperator*{\argmax}{\arg\!\max}

\newcommand{\hmM}{\widehat {\boldsymbol {\mathscr M}}}

\title{A Latent-Variable Formulation of the Poisson Canonical Polyadic Tensor Model: Maximum Likelihood Estimation and Fisher Information }

\author{
Carlos~Llosa-Vite$^1$, Daniel~M.~Dunlavy$^1$, Richard~B.~Lehoucq$^1$, \\ Oscar L\'opez$^2$, Arvind Prasadan$^1$
}
\date{%
    $^1$\textit{Sandia National Laboratories}, Albuquerque, NM and Livermore, CA\\%
    $^2$\textit{Harbor Branch Oceanographic Institute}, Florida Atlantic University, Ft Pierce, FL\\[2ex]%
}

\begin{document}

\maketitle
\abstract
We establish parameter inference for the Poisson canonical polyadic (PCP) model of tensor count data through a latent-variable formulation. 
Our approach exploits the property that any random tensor that follows the PCP model can be derived by marginalizing an unobservable random tensor of one dimension larger. 
The loglikelihood of this larger dimensional tensor, referred to as the ``complete" loglikelihood, is comprised of multiple loglikelihoods corresponding to rank one PCP models. 
Using this methodology, we first demonstrate that several existing algorithms for fitting non-negative matrix and tensor factorizations are Expectation-Maximization algorithms.  
Next, we derive the observed and expected Fisher information matrices for the PCP model by leveraging its latent-variable formulation. 
The Fisher information provides us crucial insights into the well-posedness of the tensor model, such as the role that the rank of parameter tensor plays in identifiability and indeterminacy. For the special case of PCP models with rank one parameter tensors, we demonstrate that these results are greatly simplified.

\begin{figure}[b!]
    \centering
    \includegraphics[width=.9\textwidth]{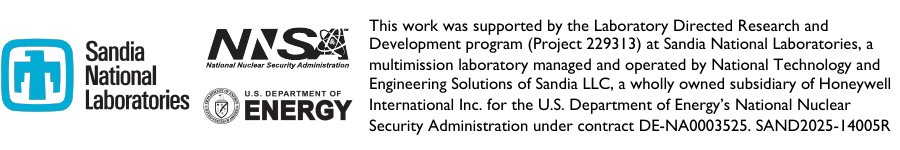}
\end{figure}

\section{Introduction}\label{sec:intro}
The contribution of our work is a formulation of the Poisson canonical polyadic (PCP) model for tensor count data as a latent variable model. In doing so, we provide new  statistical tools for parameter inference and analysis approaches applied to such models. We start by defining the PCP model and the current approach of parameter inference using maximum likelihood estimation.

A $P$-way random tensor $\ten{X}$ \emph{follows the PCP model}---or, equivalently, \emph{is PCP-distributed}---with parameter tensor $\ten{M}$ if 
\begin{subequations} \label{pcp-model-ll}
\begin{equation}\label{eq:pcp1}
x_{\bi}\sim\Poi\left(m_{\bi}\right), 
\quad m_{\bi}=\sum_{r=1}^R\prod_{p=1}^P \mat{A}_p(i_p,r) \; ,
\end{equation}
where $\ten{X} \in \mathbb{N}_{0}^{N_1\times N_2\times \dots\times N_P}$ with $\mathbb{N}_{0}$ denoting the set of natural numbers including 0, $\ten{M} \in \mathbb{R}_{+}^{N_1\times N_2\times \dots\times N_P}$ with $\mathbb{R}_{+}$ the set of positive real values, and the values $x_{\bi}$ are independent.
Here $x_{\bi}$ and $m_{\bi}$ denote the $\bi$-th entry of $\ten{X}$ and $\ten{M}$ respectively, where $\bi$ represents the multi-index vector $(i_1,i_2,\dots,i_P)$ with $i_k \in \lbrace1,2,\dots,N_k\rbrace$. The non-negative parameter tensor $\ten{M}$ is represented as a canonical polyadic (CP) model~\citep{BaKo2025} with non-negative tensor rank $R$~\citep{LiCo09} with factor matrices $\mat{A}_p \in \mathbb{R}_{+}^{N_p \times R}$ for $p \in \{1,\dots,P\}$, where $\mat{A}_p(i_p,r)$ denotes the entry of the matrix $\mat{A}_p$ in row $i_p$ and column $r$. 

The loglikelihood function of the PCP model is 
\begin{equation}\label{eq:llik1}
\ell(\bx|\btheta) = \sum_{\bi}  
\left[ x_{\bi}\log\left(\sum_{r=1}^R\prod_{p=1}^P \mat{A}_p(i_p,r)\right) - \left(\sum_{r=1}^R\prod_{p=1}^P \mat{A}_p(i_p,r)\right) -\log(x_{\bi}!) \right],
\end{equation}
where $\bx = \vecc(\ten{X})$\ is a vector of length $\prod_{p=1}^P N_p$ created using the natural ordering of $\ten{X}$~\cite{BaKo2025}, $\vecc (\mat{A}_p) \in \mathbb{R}_{+}^{N_p R} $ stacks the columns of the $p$th factor matrix of $\ten{M}$ into a vector, and
\begin{align}\label{theta-vec}
\btheta = [\vecc (\mat{A}_1)^\top\dots\vecc (\mat{A}_N)^\top]^\top \in \mathbb{R}_{+}^{(N_1 + \cdots + N_P)R}  
\end{align} 
\end{subequations}
denotes the parameter vector that contains all entries of the factor matrices of $\ten{M}$. 

Loglikelihood-based parameter inference on $\btheta$ is straightforward when $R=1$; however, when $R>1$ the presence of the sum inside the logarithm in Equation~\eqref{eq:llik1} significantly complicates the optimization and statistical analysis associated with parameter inference of the PCP model. For matrix count data, \citet{leeandseung00} used Jensen's inequality to iteratively lower bound the sum inside the logarithm in Equation~\ref{eq:llik1}, and then maximize it \citep{hunterandlange04}. Extending that approach, \citet{chiandkolda12} developed a majorization-minimization approach for estimating the parameters $\btheta$ in Equation~\ref{eq:llik1} for tensor count data. We build upon these previous advances and reformulate the PCP model as a latent-variable model, enabling us to study the PCP model through traditional frequentist-based statistical analyses. Indeed, we show in Sections~\ref{sec:genrank} and~\ref{sec:rank1} that our new latent variable model formulation leads to the first derivation of the Fisher information of the PCP model in the open literature, which allows us to prove that the model is identifiable.

The remainder of this paper is organized as follows. 
Related work and notation used in this paper are provided at the end of this section. 
In Section~\ref{sec:genrank}, we derive the PCP model with $R>1$ through its latent-variable formulation, introducing expectation-maximization (EM) algorithms to compute maximum likelihood estimators for the associated parameter inference problem.
Along the way, we demonstrate that the existing PCP model fitting algorithms of~\citet{leeandseung00} and~\citet{chiandkolda12} are instances of these (EM) algorithms. 
We then use the missing information principle of \citet{orchardandwoodbury72} to derive the observed and expected Fisher information matrices to help us determine the conditions for the identifiability of the PCP model. 
In Section~\ref{sec:rank1}, we explore parameter inference for the PCP model with $R=1$, illustrating that the latent-variable formulation and the properties of the model simplify dramatically. 
In particular, we show that the maximum likelihood estimator and expected Fisher information matrix have closed form expressions for models with $R=1$.
In Section~\ref{sec:simus} we explore the parameter estimation algorithms and the Fisher information of the PCP model empirically through various numerical experiments, illustrating the role of the Fisher information matrix in identifiability, and Barlett's identity.

\subsection{Related work and contributions}\label{sec:relatedwork}
Poisson canonical polyadic (PCP) tensor modeling has been used to uncover latent structures in multi-dimensional count data in applications such as network analysis 
\citep{dunlavyetal11a,ezicketal19,baskaranetal19,leggasetal20}, term-document analysis \citep{chewetal07,henrettyetal18}, email analysis
\citep{baderetal08}, link prediction \citep{dunlavyetal11b}, geospatial analysis \citep{henrettyetal17}, web page analysis \citep{koldaandbader06}, and differential entropy estimation~\citep{DuLeMaPr25}. 
Our work extends previous approaches to parameter inference for models of matrix count data~\citep{leeandseung00} and tensors~\citep{chiandkolda12}. Specifically, we extend the maximum likelihood estimation approach of~\citet{chiandkolda12} to the latent variable formulation of the PCP model defined in Section~\ref{sec:cpaprLatent}, which allows us to explicitly define the Fisher information of the model.

There has been a tremendous amount of effort directed toward the understanding of CP models from a statistical modeling perspective, but much of the work to date has focused on models where Gaussian noise is assumed in the data and least-squares estimators are computed. For example, \citep{sidiropoulosandbro00} established the identifiability of low-rank CP decompositions for continuous tensor data, and we study identifiability in more detail for the PCP model and count tensor data specifically through the lens of the Fisher information in Section~\ref{subsec:FIMgenrank}. \citet{liuandsidiropoulos01} established Cram\'{e}r--Rao lower bounds (CRLBs) on the variance of CP decompositions for 3rd- and 4th-order tensors, specifically for the parameters $\btheta$ in Equation~\eqref{theta-vec}, to identify variance-optimal estimators, and \citet{tichavskyetal13} extended those results to general-order tensors. \citet{sidiropoulosetal17} also provided CRLBs for the variance of estimators of $\ten{M}$, the parameter tensor of the CP decomposition. And these results have led to deeper study into the stability of decompositions~\citep{tichavskyandkoldovsky11}, optimal decompositions for structured tensors~\citep{boizardetal2015}, coupled decompositions~\citep{prevostetal22}, and a wide range of applications~\citep[e.g.,][]{decurningeetal21, rakhimovetal21}. \citet{fariasandcomon15} illustrated similar advances for continuous tensor data with various non-Gaussian noise distributions as well, but their problem formulations cannot be leveraged directly for the PCP models. As the inverse of the Fisher information matrix is a critical component in the computation of CRLBs, our work on the Fisher information of the PCP model in Sections~\ref{subsec:FIMgenrank} and~\ref{subsec:FIM1} is a key initial step in establishing a CRLB for estimators of the PCP model. Although we establish the Fisher information in this work, we leave the development of the latter as future work.

\citet{huanganssidiropoulos17} developed a closed-form solution of a CP decomposition with parameter tensor having non-negative tensor rank of one using Kullback-Leibler (KL) divergence, similar to loss function used by \citet{leeandseung00}. They also extended their ideas to the case of parameter tensors with non-negative tensor rank greater than one to establish an expectation-maximization-like iterative algorithm for fitting KL divergence-based CP models, similar to the approach of \citet{chiandkolda12}. That work is most similar to our work presented here, although we extend the advances presented there by explicitly establishing a latent variable model for and Fisher information of the PCP model.

\subsection{Notation}
\label{sec:notation}

The notation for the PCP model was introduced in Equations~\eqref{eq:pcp1}-\eqref{theta-vec} and, for convenience, we also provide additional notation that is used throughout the remainder of this paper. Although we follow standard tensor notation defined in~\citep{BaKo2025}, several deviations and extensions are defined below to simplify the discussions related to the latent-variable formulation of the PCP model and its various properties. 

Scalar values are denoted by lowercase letters (e.g., $x$), vectors are denoted as bold lowercase letters (e.g., $\bx$ and $\blambda$), matrices are denoted as
bold uppercase letters (e.g., $\mat{X}$), and general tensors are denoted as bold script
letters (e.g., $\ten{X}$). Dimensions of tensors are denoted as uppercase letters (e.g., $P$) and index sets are denoted as $[P] = \{1,\dots,P\}$. In the setting of the PCP model, count data tensors are defined over $\mathbb{N}_{0}$, the set of natural numbers, including 0, and parameter tensors are defined over $\mathbb{R}_{+}$, the set of positive real values. When defining notation or discussing tensors and tensor properties in general, we define tensors over $\mathbb{R}$, the set of real values, which include both $\mathbb{N}_{0}$ and $\mathbb{R}_{+}$ as subsets. 

A non-negative rank-one $P$-way tensor is defined as the outer product of $P$ vectors containing only non-negative elements. The \emph{non-negative rank} of a non-negative tensor is the smallest number of non-negative rank one tensors whose sum equals the original tensor~\citep{LiCo09}. In this work, we refer to the non-negative tensor rank of a non-negative tensor simply as the \emph{rank} of the tensor unless the context is unclear.

Let $N = N_1 N_2 \cdots N_P$ denote the product of the $P$ dimensions of the tensor $\ten{X} \in \mathbb{R}^{N_1 \times N_2 \times \cdots \times N_P} $.
The quantities $\vecc(\ten{X})\in \mathbb{R}^{N}$ and $\mat{X}_{(p)} \in \mathbb{R}^{N_p \times N/N_p} $ denote the vectorization of $\ten{X}$
and the matricization of the $p$th mode of the tensor $\ten{X}$, respectively. By $\dv(\ten{X})$ we denote the $N\times N$ diagonal matrix with $\vecc(\ten{X})$ in its diagonal, so that $\dv(\ten{X})\bone_{N} = \vecc(\ten{X})$, where $\bone_{N}$ is a vector of length $N$ containing all ones. 
Whenever $\ten{X}$ and $\ten{Y}$ are tensors of the same size, $\ten{X}*\ten{Y}$ and $\ten{X}\oslash\ten{Y}$ denote their element-wise (i.e., Hadamard) product and element-wide division, respectively. Similarly, we denote $\ten{X}^{*k}$ as the $k$th Hadamard power, which is applied element-wise to the values in the tensor. In addition, $\log(\mat{X})$ and $\log(\ten{X})$ apply the logarithm function element-wise to matrices and tensors, respectively.
The mode-$k$ product of tensor $\ten{X}$ and vector $\ba$ (which has length equal to $N_k$, the size of dimension $k$ of $\ten{X}$), also called a contraction of $\ten{X}$ along mode $k$, is denoted as $\ten{X}\bar\times_k\ba$ and results in a tensor of one dimension less than $\ten{X}$.
Leveraging this notation, we denote 
$
\ten{X} \bar\times_{q\neq k} (\ba_{q})
$
as the contraction of all but the $k$th mode of $\ten{X}$ with the vectors $\ba_q\in\mbbR^{N_q}$, so that the result is a vector of size $N_k$. Similarly, when $P>2$ we denote 
$
\ten{X} \bar\times_{q\neq k,l} (\ba_{q})
$
as the contraction of all but the $k$th and $l$th modes of $\ten{X}$ with the vectors $\ba_q$, so that the result is a matrix of size $N_k\times N_l$ when $k<l$.

The Kronecker product of the vectors $\vct{a} \in \mathbb{R}^m$ and $\vct{b} \in \mathbb{R}^n$ is denoted by $\vct{c} = \vct{a} \otimes \vct{b} \in \mathbb{R}^{mn}$ with $c_k = a_ib_j$ where $k = (i-1)n + j$. 
The Hadamard product of matrices $\mat{A} \in \mathbb{R}^{m \times r}$ and $\mat{B} \in \mathbb{R}^{m \times r}$ is denoted by $\mat{C} = \mat{A} * \mat{B} \in \mathbb{R}^{m \times r}$ with $\mat{C}(i, j) = \mat{A}(i, j) * \mat{B}(i, j)$ for all $i \in \{1, \dots, m\}$ and $j \in \{1, \dots, r\}$. 
In contrast, the Khatri-Rao product, or Kronecker product in columns, of the matrices $\mat{A}\in \mat{R}^{m \times r}$ and $\mat{B}\in \mathbb{R}^{n \times r}$ is denoted by $\mat{C} = \mat{A}\odot \mat{B} = \left[\vct{a}_{1} \otimes \vct{b}_{1} \quad \vct{a}_{2} \otimes \vct{b}_{2} \quad \cdots \quad \vct{a}_{r} \otimes \vct{b}_{r}\right] \in \mathbb{R}^{mn \times r}$, where $\vct{a}_{i}$ and $\vct{b}_{i}$ are the $i$th columns of $\mat{A}$ and $\mat{B}$, respectively. 
We also apply these various operators to ordered indexed sets of vectors, matrices, or tensors. 
For example, $\odot\mat{A}_{[P]}  = \mat{A}_P \odot \cdots \odot \mat{A}_1$  
is the Khatri-Rao product of $P$ matrices with indices in $[P]$ in \emph{reverse} order. More generally, the subscript $[P]$ can be any ordered set, and the products are computed with indices from that ordered set in decreasing order. Notably, we make use of following: $\odot\mat{A}_{[P]\setminus\{k\}}  = \mat{A}_P \odot \cdots \odot \mat{A}_{k+1} \odot \mat{A}_{k-1} \odot \cdots \odot \mat{A}_1$ and $\odot\mat{A}_{[P]\setminus\{k, l\}}  = \mat{A}_P \odot \cdots \odot \mat{A}_{l+1} \odot \mat{A}_{l-1} \odot \cdots \odot \mat{A}_{k+1} \odot \mat{A}_{k-1} \odot \cdots \odot \mat{A}_1$, where $k<l$. 
Similarly, we make use of the following Hadamard product $*\blambda_{[P]} = \blambda_{1} * \cdots * \blambda_{P}$. Because the Hadamard product is commutative, in this case the set $[P]$ does not need to be ordered. Notable examples are $*\blambda_{[P]\setminus\{k\}} = \blambda_{1} * \cdots * \blambda_{k-1} * \blambda_{k+1} * \cdots * \blambda_{P}$ and $*\blambda_{[P]\setminus\{k, l\}} = \blambda_{1} * \cdots * \blambda_{k-1} * \blambda_{k+1} * \cdots * \blambda_{l-1} * \blambda_{l+1} * \cdots * \blambda_{P}$ with $k<l$.
Such vectors appear in this work as the sums of columns of the factor matrices of CP decompositions, which we denote as $\blambda_q = \mathbf{A}_q^\top\bone_{N_q} \in \mathbb{R}^{R}$ for $\mathbf{A}_q \in \mathbb{R}^{N_q \times R}$.

\section{PCP Inference}\label{sec:genrank}
In this section, we introduce parameter inference of the PCP model when the rank of the parameter tensor $R$ is larger than one and defer the special case of $R=1$ to Section~\ref{sec:rank1}, as the latter case leverages several key simplifications with respect to parameter inference.

As noted previously, the sum inside the logarithms of the logarithmic likelihood function of the PCP model in Equation~\eqref{eq:llik1} entails a challenging parameter inference problem. 
Our approach is to treat the PCP model as a latent-variable model, which leads to a complete loglikelihood $\ell_c$ that does not contain a sum inside the two logarithms. 
However, $\ell_c$ cannot be used for parameter inference directly because it involves unobserved data, and so we first derive the conditional expectation of the complete loglikelihood $\ell_c$ in Section~\ref{sec:cpaprLatent}. This quantity is then used to calculate the maximum likelihood estimators in Section~\ref{subsec:EM} and to obtain the observed and expected Fisher information matrices in Section~\ref{subsec:FIMgenrank}.

\subsection{Latent variable formulation}\label{sec:cpaprLatent}

Our latent variable is a tensor $\ten{Z} \in \mathbb{N}_{0}^{R\times N_1\times N_2\times\dots\times N_P} $ of dimension one larger than $\ten{X}$ as defined in Equation~\eqref{eq:pcp1} where
\begin{subequations}
\begin{equation}\label{eq:latent1}
z_{r,\bi} \sim \Poi\left(\prod_{p=1}^P \mat{A}_p(i_p,r)\right) \text{for } r=1,\ldots, R\,.
\end{equation}
Summing over the first dimension of the latent tensor $\ten{Z}$ implies
\begin{equation}\label{eq:latent2}
\sum_{r=1}^R z_{r,\bi} \sim \Poi\left(\sum_{r=1}^R\prod_{p=1}^P \mat{A}_p(i_p,r)\right),
\end{equation}
\end{subequations}
because the sum of independent Poisson distributed random variables is also Poisson distributed, and
imply that marginalizing $\ten{Z}$ in the first dimension
matches the distribution of $x_{\bi}$ in Equation~\eqref{eq:pcp1}.
In contrast, this implies that a PCP-distributed tensor $\ten{X}$ can be understood as arising from a latent tensor $\ten{Z}$ of a dimension one larger than $\ten{X}$.
In such a context, Orchard and Woodbury state that a \textit{relatively simple analysis is transformed into a complex one just because some of the information is missing}~\citep{orchardandwoodbury72}. In our case, if $\ten{Z}$ were observed, the parameter inference for $\btheta$ would be simpler. This is because the loglikelihood for \eqref{eq:latent1},
\begin{equation}\label{eq:em_completell}
\ell_c(\bz|\btheta)  = \sum_{r,\bi} \left[z_{r,\bi}\log\left(\prod_{p=1}^P \mat{A}_p(i_p,r)\right)
-\left(\prod_{p=1}^P \mat{A}_p(i_p,r)\right)   -\log(z_{r,\bi}!)
\right] \; ,
\end{equation}
does not contain a sum inside the two logarithms in contrast to the likelihood \eqref{eq:llik1}. 
Unfortunately, the complete loglikelihood $\ell_c$ involves entries of the unobserved tensor $\ten{Z}$.
In this article, we leverage the standard tools developed to perform parameter inference for $\btheta$ that are based on the complete loglikelihood $\ell_c$ instead of $\ell$.
Each $m_{\bi}$ in the likelihood \eqref{eq:llik1} represents a parameter of a Poisson distribution; therefore, we require that all $m_{\bi}$
and all entries of $\mat{A}_p$ in the latent variable formulation \eqref{eq:latent1} are positive so that every entry in $\btheta$ is also positive.

Now that we have an explicit representation for the complete loglikelihood, we review the standard latent-variable model for vector data and latent variables. 
Let the tensors $\ten{X}$ and $\ten{Z}$ be defined as in Equations \eqref{eq:pcp1} and \eqref{eq:latent1}. We call $\bx = \vecc(\ten{X})$ the \textit{observed-data vector} and $\bz=\vecc(\ten{Z})$ the \textit{complete-data vector}. 
Using the law of conditional probability, we can write the joint distribution of $\bx$ and $\bz$ as $p(\bz,\bx|\btheta) = p(\bx|\btheta)p(\bz|\bx,\btheta)$. Taking the logarithm on both sides and using $p(\bz,\bx|\btheta) = p(\bz|\btheta)$ (which holds because $\bx$ is fully determined by $\bz$) leads to the following loglikelihood decomposition
\begin{subequations}
\begin{equation}\label{eq:llik_completemiss}
\ell(\bx|\btheta) = \ell_c(\bz|\btheta) - \ell_m(\bz|\bx,\btheta),
\end{equation}
where
\begin{itemize}
\item $\ell(\bx|\btheta) \coloneqq \log\ p(\bx|\btheta)$ denotes the \textit{loglikelihood} in Equation \eqref{eq:llik1},
\item $\ell_c(\bz|\btheta) \coloneqq \log\ p(\bz|\btheta)$ denotes the \textit{complete loglikelihood} in Equation \eqref{eq:em_completell}, and
\item $\ell_m(\bz|\bx,\btheta) \coloneqq \log\ p(\bz|\bx,\btheta)$ denotes the \textit{missing loglikelihood}.
\end{itemize}

For PCP models with parameter tensors of rank one, i.e., Equation~\eqref{eq:pcp1} with $R=1$, there are no latent variables; thus,  $\ten{Z}=\ten{X}$ and $\ell(\bx|\btheta) = \ell_c(\bz|\btheta)$ because $p(\bx|\bx,\btheta)=1$. 
This special case is studied in detail in Section \ref{sec:rank1}.

Equation \eqref{eq:llik_completemiss} divides the loglikelihood $\ell$ into the complete loglikelihood $\ell_c$ and missing loglikelihood $\ell_m$, which both involve unobserved data $\bz$. Since $\bz$ is not observed, we can take the conditional expectation of both sides of \eqref{eq:llik_completemiss} with respect to the conditional distribution of $\bz$ given  $\bx$, which leads to the conditional expectation version of the likelihood decomposition \eqref{eq:llik_completemiss}
\begin{equation}\label{eq:llik_completemiss2}
\ell(\bx|\btheta) = \underbrace{\mbbE[\ell_c(\bz|\btheta) | \bx,\bar\btheta ]}_{Q(\btheta,\bar\btheta)} - \underbrace{\mbbE[\ell_m(\bz|\bx,\btheta) | \bx,\bar\btheta ]}_{H(\btheta,\bar\btheta)},
\end{equation}
\end{subequations}
where the equality holds for any value of $\bar\btheta$ contained in the same parameter space as $\btheta$ and we used the identity
$$
\mbbE[\ell(\bx|\btheta) | \bx,\bar\btheta ] \coloneqq
\int
\ell(\bx|\btheta) p(\bz|\bx,\bar\btheta) \, \text{d}\bz
= 
\ell(\bx|\btheta)
\left[\int 
 p(\bz|\bx,\bar\btheta) \,\text{d}\bz\right]
  = 
  \ell(\bx|\btheta)\,.
$$

The expected complete loglikelihood $Q(\btheta,\bar\btheta)$  of Equation \eqref{eq:llik_completemiss2} is a fundamental quantity in latent-variable models because it enables us to perform likelihood-based inference without using the loglikelihood function. The next result determines $Q(\btheta,\bar\btheta)$ for our PCP latent-variable model. 
The proof is elementary and depends upon the relationship between Poisson and multinomial distributions. We include it here for sake of completeness and because many in the tensor research community may be unfamiliar with these probabilistic results. 

\begin{lemma}\label{thm:Qfunc}
Consider the latent-variable formulation of the PCP model in Equations~\eqref{eq:pcp1}-\eqref{theta-vec} with $\btheta = [\vecc(\mat{A}_1)^\top\dots\vecc(\mat{A}_P)^\top]^\top$, $\bar\btheta = [\vecc(\mathbf{\bar A}_1)^\top\dots\vecc(\mathbf{\bar A}_P)^\top]^\top$, and the notation of \S \ref{sec:notation}.
Then the expected complete loglikelihood $Q(\btheta,\bar\btheta)$ is
\begin{equation}\label{eq:Qfunc1}
    Q(\btheta,\bar\btheta)  = \sum_{r,\bi} \left[\bar z_{r,\bi}\log\left(\prod_{p=1}^P \mat{A}_p(i_p,r)\right)-
\left(\prod_{p=1}^P \mat{A}_p(i_p,r)\right)  
\right]-C_1(\bar\btheta) ,
\end{equation}
where  $C_1(\bar\btheta) = \sum_{r,\bi}\mbbE_{\bz|\bx,\bar\btheta}\log(z_{r,\bi}!)$ does not depend on $\btheta$ and
$$
\bar z_{r,\bi} \coloneqq
\mbbE_{\bz|\bx,\bar\btheta} \left(z_{r,\bi}\right) 
=
\left(\prod_{p=1}^P \bar{ \mat{A}}_p(i_p,r)\right)
\dfrac{x_{\bi}}
{\sum_{r=1}^R \prod_{p=1}^P \bar{ \mat{A}}_p(i_p,r)}.
$$

\end{lemma}
\begin{proof}
The proof follows from the definition of $Q(\btheta,\bar\btheta)$ in Equation \eqref{eq:llik_completemiss2}, and the complete loglikelihood $\ell_c(\bz|\btheta)$ in Equation \eqref{eq:em_completell}.  To find $\bar z_{r,\bi}$, first let $\bz_{\bi} = [z_{1,\bi}, \cdots,z_{R,\bi}]^\top$. Next, we will identify the random variable $\bz_{\bi}|x_{\bi}$. This is the random variable that we will take our expectation over when finding $Q(\btheta,\bar\btheta)$, and its loglikelihood, added over all $\bi$, corresponds to the missing loglikelihood $\ell_m(\bz|\bx,\btheta)$ of Equation \eqref{eq:llik_completemiss}. It follows that 
\begin{equation*}\label{eq:emconditional2}
\begin{aligned}
p(\bz_{\bi}|x_{\bi})  =
 \dfrac{p(\bz_{\bi},x_{\bi})}{p(x_{\bi})}
 =
 \dfrac{p(\bz_{\bi})}{p(x_{\bi})}
&=
\dfrac{
\prod_{r=1}^R (z_{r,\bi}!)^{-1}\left(\prod_{p=1}^P \mat{A}_p(i_p,r)\right)^{z_{r,\bi}}
}{
(x_{\bi}!)^{-1}\left(\sum_{r=1}^R \prod_{p=1}^P \mat{A}_p(i_p,r)\right)^{x_{\bi}}
}
\\&=
\frac{x_{\bi}}{z_{1,\bi}!\dots z_{R,\bi}!}
(p_{1,\bi})^{z_{1,\bi}}
\dots 
(p_{R,\bi})^{z_{R,\bi}}
\end{aligned}
\end{equation*}
is the probability mass function (PMF) of a multinomial distribution with number of trials $x_{\bi}$ and mutually exclusive event probabilities $p_{1,\bi},\dots,p_{R,\bi}$ where $\prod_{r=1}^R (z_{R,\bi}!)^{-1}/\big(x_{\bi}!\big)^{-1}=x_{\bi}/(z_{1,\bi}!\dots z_{R,\bi}!)$ and $ p_{r,\bi} = \prod_{p=1}^P \mat{A}_p(i_p,r)/(\sum_{r=1}^R \prod_{p=1}^P \mat{A}_p(i_p,r))$. 
Hence we write
\begin{equation*}\label{eq:conditionalmulti}
\bz_{\bi}|x_{\bi}
\sim 
\text{Multinomial}(x_{\bi} ; p_{1,\bi},\dots,p_{R,\bi})\,,
\end{equation*}
which implies each component of $\bz_{\bi}|x_{\bi}$ has mean $\bar z_{r,\bi}=x_{\bi} \bar p_{r,\bi}$.
\end{proof}

The expected complete loglikelihood $Q(\btheta,\bar\btheta)$ of Lemma \ref{thm:Qfunc} will be used in the remainder of this section to perform maximum likelihood estimation, and to obtain the Fisher information matrix of the PCP model.

\subsection{Maximum likelihood estimation }\label{subsec:EM}

We can compute the maximum likelihood estimator of the PCP model by optimizing $\ell(\bx|\btheta)$, which is a nonlinear nonconvex function. 
The next result reformulates Lemma \ref{thm:Qfunc} in terms of matrices, revealing that $Q(\btheta,\bar\btheta)$ is a sum of $P$ concave subproblems, which  leads to a simpler optimization problem than maximizing $\ell(\bx|\btheta)$.

\begin{theorem}\label{thm:Qfunc2}
   The expected complete loglikelihood $Q(\btheta,\bar\btheta)$ of Equation \eqref{eq:Qfunc1} can be equivalently written as
\begin{subequations}
\begin{align}
Q(\btheta,\bar\btheta)  
 &= \bone_{N_p}^\top \Big[
\mathbf{\bar Z}_p*\log\left(\mathbf{A}_p\right)
 -
 \mathbf{A}_p\diag(*\blambda_{[P]\setminus\{p\}})
\Big]\bone_{R}
+C_2(\bar\btheta,\btheta_{[P]\setminus\{p\}}), \label{eq:Qfunc2} 
\end{align}
which holds for any for $p\in [P]$, and
\begin{equation}\label{eq:Zmat}
\mathbf{\bar Z}_p =
 \mathbf{\bar A}_p * \left([\mat{X}_{(p)}\oslash ( \mat{\bar A}_p(\odot\mathbf{\bar A}_{[P]\setminus\{p\}})^\top)] \odot\mathbf{\bar A}_{[P]\setminus\{p\}}\right).\, 
\end{equation}
\end{subequations}
The term
\begin{equation}
C_2(\bar\btheta,\btheta_{[P]\setminus\{p\}}) =   \sum_{q=1, q\neq p}^P\bone_{N_q}^\top  \left(\bar{\mat{Z}}_q*\log\left(\mat{A}_q\right)\right)\bone_{R}-C_1(\bar\btheta)
\end{equation}
does not depend on $\mathbf{A}_p$. Furthermore, $Q(\btheta,\bar\btheta)$ is concave with respect to $\mathbf{A}_{p}$ when factor matrices $\mathbf{A}_{1}, \dots, \mathbf{A}_{p-1}, \mathbf{A}_{p+1}, \dots, \mathbf{A}_{P}$ are fixed.
\end{theorem}
\begin{proof}
This proof follows from 
$\log(\prod_{p=1}^P \mat{A}_p(i_p,r)) = \sum_{p=1}^P\log\left( \mat{A}_p(i_p,r)\right)$, and recognizing that \newline
$\sum_{r,\bi} 
\prod_p \mat{A}_p(i_p,r)  = \blambda_p^\top (*\blambda_{[P]\setminus\{p\}})$ for any $p\in [P]$. The tensor $\mbbE_{\bz|\bx,\bar\btheta} \left(\ten{Z}\right)$ with entries $\bar z_{r,\bi}$ has matrix form $\mathbf{\bar Z}_p$ after contracting $P-1$ of its modes:
$$
\mathbf{\bar Z}_p(i_p,r) = 
\sum_{i_{1}=1}^{N_1} \cdots 
\sum_{i_{p-1}=1}^{N_{p-1}}
\sum_{i_{p+1}=1}^{N_{p+1}} \cdots
\sum_{i_{P}=1}^{N_P}
\bar z_{r,\bi}\; .
$$

\end{proof}

As mentioned in Section~\ref{sec:relatedwork}, CP models include an an inherent indeterminacy due to a choice needed for the columns of the factor matrices. Corollary~\ref{cor:Qfunc} shows that the standard choice used in existing CP parameter inference algorithms~\cite[see, e.g.,]{chiandkolda12} leads to a useful interpretation of the expression for $Q(\btheta,\bar\btheta)$ in Theorem~\ref{thm:Qfunc}.

\begin{corollary}\label{cor:Qfunc}
 If $*\blambda_{[P]\setminus\{p\}} = \bone_R$ for any $p \in [P]$---i.e., the columns of $\mat{A}_p$ contain the weights of all the multilinear products in Equation~\eqref{eq:pcp1}, then $Q(\btheta,\bar\btheta)$ is identical to the loglikelihood of $\mat{\bar Z}_p \sim \Poi(\mat{A}_p)$ up to a constant.
\end{corollary}
\begin{proof}
The loglikelihood of $\mat{\bar Z}_p \sim \Poi(\mat{A}_p)$ is
$$
\bone_{N_p}^\top \Big[
\mathbf{\bar Z}_p*\log\left(\mathbf{A}_p\right)
 -
 \mathbf{A}_p
\Big]\bone_{R} + C,
$$
where $C$ does not depend on $\mat{A}_p$. This corresponds to Equation \eqref{eq:Qfunc2} (up to a constant) after noting that the term $\diag(*\blambda_{[P]\setminus\{p\}})$ identity when $*\blambda_{[P]\setminus\{p\}} =\bone_R$.
\end{proof}

A consequence of Corollary~\ref{cor:Qfunc} is that with this particular scaling of the factor matrices, parameter inference for $\mat{A}_p$ is straight forward using $Q(\btheta,\bar\btheta)$ than with the loglikelihood in Equation~\eqref{eq:llik1}
With the closed-form expression for $Q(\btheta,\bar\btheta)$ in Equation \eqref{eq:Qfunc2}, we now use it for maximum likelihood estimation.

\subsubsection{Expectation-maximization}\label{sec:EM}
Expectation-maximization (EM) algorithms \citep{dempsteretal77,mclachlanandthri07}  are a class of algorithms that are useful when optimizing the loglikelihood $\ell$ is more challenging than optimizing $Q(\btheta,\bar\btheta)$, a situation common in latent-variable models. Starting with Equation~\eqref{eq:llik_completemiss2}, we can demonstrate through Jensen's inequality that
\begin{equation}\label{eq:emalg1}
\ell(\bx|\btheta)
\geq
Q(\btheta,\bar\btheta)  - H(\bar\btheta;\bar\btheta) 
\end{equation}
holds for any $(\btheta,\bar\btheta)$. 
Equality holds when $\btheta = \bar\btheta$ and is trivially obtained from Equation \eqref{eq:llik_completemiss2}. Hence, the right-hand side of Equation \eqref{eq:emalg1} ``minorizes" the loglikelihood and depends on $\btheta$ only through $Q(\btheta,\bar\btheta)$, while the missing loglikelihood component $H(\bar\btheta;\bar\btheta)$ is constant for all $\btheta$. In this context, ``minorizes" refers to the concept used in MM (i.e., Majorize-Minimization or Minorize-Maximization) algorithms \citep{hunterandlange04}. An MM algorithm iteratively optimizes a difficult objective function (in this case $\ell$) by optimizing a simpler surrogate function that minorizes the original function (in this case $Q(\btheta,\bar\btheta)$). EM algorithms are MM algorithms that use $Q(\btheta,\bar\btheta)$ as a minorizing function to optimize $\ell(\bx|\btheta)$. 

An EM algorithm iteratively updates the estimated parameters of a statistical model by alternating between expectation (E) and maximization (M) steps. At each iteration $t$, $Q(\btheta,\bar\btheta)$ is updated using $\bar\btheta = \btheta^{(t-1)}$ during the E-step, which is then optimized in the M-step to find the new parameter estimates $\btheta^{(t)}$. That is, an EM algorithm iteratively updates $\btheta^{(t)} \leftarrow \argmax_{\btheta}Q(\btheta;\btheta^{(t-1)})$ until convergence.

\subsubsection{Expectation-conditional maximization (ECM)}\label{sec:MLEgenrank}

An expectation-conditional maximization (ECM) algorithm \citep{mengandrubin93} is a variant of EM where the parameter vector $\btheta$ is split into multiple blocks. The E-step is unchanged, but in the M-step, a series of conditional maximization (CM) steps over each block are taken, while the remaining blocks are fixed at their estimated values from the previous iteration. ECM is attractive for use with $Q(\btheta,\bar\btheta)$ in  Equation~\eqref{eq:Qfunc2}, as it is parameterized in terms of each factor matrix $\mat{A}_p$, with $p \in [P]$. Hence, each factor matrix can be assigned its own block.  Below are the details of the steps performed at iteration $t$ for this ECM algorithm.

\paragraph{E-step:} 
The E-step involves updating $Q(\btheta,\bar\btheta)$ at $\bar\btheta = \btheta^{(t-1)}$, i.e., updating the terms $\bar z_{r,\bi}$ in Equation \eqref{eq:Qfunc1} using the parameters estimated after the previous iteration $t-1$. In each iteration of a standard ECM algorithm, $P$ CM-steps will take place after one $E$ step. 

\paragraph{CM-steps:} 
The $p$-th CM-step, where $p\in [P]$,
consists of optimizing $Q(\btheta;\btheta^{(t-1)})$ in Equation \eqref{eq:Qfunc2} over $\mat{A}_p$ with $\mat{\bar Z}_p$ fixed and evaluated at $\btheta^{(t-1)}$. As stated in Theorem \ref{thm:Qfunc2}, \eqref{eq:Qfunc2} is concave, and has a global maxima at 
\begin{equation}\label{eq:CMstep1}
\mat{A}_p^{(t)}  = \mat{\bar Z}_p^{(t-1)}\diag(*\blambda_{[P]\setminus\{p\}}^{(t-1)})^{-1}.
\end{equation}
As per Corollary \ref{cor:Qfunc},  $\diag(*\blambda_{[P]\setminus\{p\}}^{(t-1)})$ reduces to an identity matrix when all the weights across columns of the factor matrices are shifted to the columns of $\mat{\bar A}_p^{(t-1)}$, which is standard practice when computing low-rank CP decompositions (see, e.g.,~\citep{chiandkolda12}).

\subsubsection{Multi-cycle ECM (MCECM)} 
A multi-cycle ECM (MCECM) algorithm \citep{mengandrubin93} allows for multiple E-steps to be performed per iteration. A cycle is defined to be one E-step followed by one CM-step, and $P$ cycles are performed per iteration. Below are the details of the steps performed during cycle $p \in [P]$ of $t$ for the MCECM algorithm.

\paragraph{E-step:}
The E-step involves updating $Q(\btheta,\bar\btheta)$ at $\bar\btheta = \btheta_{p}^{(t-\frac12)}$. Following notation from~\citet{mengandrubin93}, we define \begin{align}\label{eq:theta_MCECM}
\btheta_{p}^{(t-\frac12)} = [
\vecc(\mathbf{A}_1^{(t)})^\top 
\cdots 
\vecc(\mathbf{A}_{p-1}^{(t)})^\top 
\vecc(\mathbf{A}_{p}^{(t-1)})^\top 
\cdots 
\vecc(\mathbf{A}_P^{(t-1)})^\top
]^\top \; ,
\end{align}
i.e., the parameter vector after $p-1$ cycles have been performed, which uses the factors matrices already updated during the current iteration $t$ and those from the previous iteration $t-1$ otherwise.

\paragraph{CM-step:}
The CM-step consists of iteratively optimizing $Q(\btheta;\btheta_{p}^{(t-\frac12)})$ over $\mat{A}_p$ with the values in all of the other factor matrices fixed. Similar to  Equation~\eqref{eq:theta_MCECM}, we define
\begin{subequations}
\begin{align}\label{eq:odotminus_MCECM}
\odot\mat{A}_{[P]\setminus\{p\}}^{(t-\frac12)} = \mat{A}_{1}^{(t)} \odot \cdots \odot \mat{A}_{p-1}^{(t)} \odot \mat{A}_{p+1}^{(t-1)} \odot \cdots \odot \mat{A}_{P}^{(t-1)}, 
\end{align}
and
\begin{multline} \label{eq:lambda_MCECM}  
	*\blambda_{[P]\setminus\{p\}}^{(t-\frac12)}  = 
	\left(\mathbf{A}_{1}^{(t)}\right)^\top\bone_{N_1} * 
     \cdots * 
	\left(\mathbf{A}_{p-1}^{(t)}\right)^\top\bone_{N_{p-1}}  * \left(\mathbf{A}_{p+1}^{(t-1)}\right)^\top\bone_{N_{p+1}} * 
	\cdots * 
	\left(\mathbf{A}_{P}^{(t-1)}\right)^\top\bone_{N_P} ,
\end{multline}
which use the factor matrices already updated during the current iteration $t$ and those from the previous iteration $t-1$.
Combining Equations~\eqref{eq:CMstep1} and \eqref{eq:Zmat}, along with the MCECM updates in Equations~\eqref{eq:odotminus_MCECM} and~\eqref{eq:lambda_MCECM}, leads to the following CM-step iterative update
\begin{equation}\label{eq:MCECM}
\mat{A}_p^{(t, i)}  = \Big[\mat{A}_p^{(t, i-1)} * \left([\mat{X}_{(p)}\oslash (\mat{A}_p^{(t,i-1)} (\odot\mat{A}_{[P]\setminus\{p\}}^{(t-\frac12)})^\top)] \odot\mat{A}_{[P]\setminus\{p\}}^{(t-\frac12)}\right) \Big]\diag(*\blambda_{[P]\setminus\{p\}}^{(t-\frac12)})^{-1} \; ,
\end{equation}
\end{subequations}
where $\mat{A}_p^{(t, i)}$ defines the $i$-th CM-step iterate of the $p$-th cycle of iteration $t$ of the MCECM algorithm. Note that the CM-step iterations for computing $\mat{A}_p^{(t)}$ start with $\mat{A}_p^{(t, 0)} = \mat{A}_p^{(t-1)}$.
Similar to Equation~\eqref{eq:CMstep1}, the last term in Equation~\ref{eq:MCECM} reduces to identity if we choose the weights of $\mat{A}_p^{(t, i)}$ and $\odot\mat{A}_{[P]\setminus\{p\}}^{(t-\frac12)}$ appropriately.

\subsubsection{Recasting existing algorithms as MCECM algorithms}\label{subsec:popem}

The update in Equation \eqref{eq:MCECM} is identical to those used by \citet{leeandseung00} for \emph{nonnegative matrix factorization} (NMF) (i.e., $P=2$) and by \citet{chiandkolda12} for {canonical polyadic alternating Poisson regression} (CP-APR) tensor decompositions (i.e., for general $P$). Hence, those algorithms can be viewed as MCECM algorithms, differing only in $P$ and the order of which the cycles are iterated. 

For the $P=2$ case (i.e., NMF), $\odot\mat{A}_{[P]\setminus\{1\}} = \mat{A}_2$ and $\odot\mat{A}_{[P]\setminus\{2\}} = \mat{A}_1$, and hence the algorithm of \citet{leeandseung00} updates \eqref{eq:MCECM} in the following order
$$
\mat{A}_1^{(1,1)}\rightarrow \mat{A}_2^{(1,1)} \rightarrow \mat{A}_1^{(2,1)}\rightarrow \mat{A}_2^{(2,1)}\ \rightarrow \dots \ .
$$

For the general $P$ case (i.e., CP-APR) , calculating each $\odot\mat{A}_{[P]\setminus\{p\}}$ can incur significant computational cost for large tensor dimensions $P$ and/or dimension sizes $N_p$. For this reason, the algorithm of \citet{chiandkolda12} involves evaluating Equation \eqref{eq:MCECM} multiple times for each value of $p$, all using the same value of $\odot\mat{A}_{[P]\setminus\{p\}}^{(t-\frac12)}$. Hence, the algorithm of \citet{chiandkolda12} updates \eqref{eq:MCECM} in the following order
$$
\underbrace{\mat{A}_1^{(1,1)}\rightarrow\dots \rightarrow \mat{A}_1^{(1,k)} }_{\text{update }\mat{A}_1}
\rightarrow 
\underbrace{\mat{A}_2^{(1,1)}\rightarrow\dots \rightarrow \mat{A}_2^{(1,k)} }_{\text{update }\mat{A}_2}
\rightarrow \dots \rightarrow 
\underbrace{\mat{A}_P^{(1,1)}\rightarrow\dots \rightarrow \mat{A}_P^{(1,k)} }_{\text{update }\mat{A}_P} 
\rightarrow \dots
$$
for some specified maximum number of CM-step iterations $k$.
Hence, \citep{leeandseung00} is the $P=2$ special case of \citep{chiandkolda12}, where Equation \eqref{eq:MCECM} is only evaluated once in each ECM cycle. 

Although we identify NMF and CP-APR as MCECM algorithms, neither were derived as such nor were explicitly associated with a latent-variable formulation of a Poisson model as in Equation~\eqref{eq:latent2}. With this new derivation, both NMF and CP-APR inherit the properties of MCECM algorithms---most notably the convergence proof derived by~\citet{mengandrubin93} and the specification of the Fisher information of the underlying statistical models associated with these algorithms, which we derive in the next section.

\subsection{Fisher information}\label{subsec:FIMgenrank}

Fisher information is a fundamental quantity in statistics that quantifies the amount of information that can be inferred about an unknown parameter of a statistical model. 
There are two forms of Fisher information: the observed Fisher information and the expected Fisher information.
The observed Fisher information 
$$
\FIM_{obs}(\btheta,\bx) \coloneqq -\nabla^2_{\btheta,\btheta}\ell(\btheta).
$$
is calculated using the data and measures the information that an instance of observed data $\bx=\vecc(\ten{X})$ provides about the parameter $\btheta$, and its sensitivity to changes in the parameter based on the data.
In contrast, the expected Fisher information 
$$
\FIM(\btheta) \coloneqq \mathbb{E}\left[ \FIM_{obs}(\btheta,\bx)\right],
$$
reflects the information that is obtained from the entire population according to the model where the expectation is taken with respect to $\bx$ under the model parameterized by $\btheta$. 
The expected Fisher information does not depend upon the observed data but rather on the theoretical distribution of the random variable. 
Hence, it provides insight into the potential information that could be gleaned about the parameter from the model.

Similar to the problem of finding maximum likelihood estimators, finding the Fisher information matrix for the rank one case is easier than for the general rank because it corresponds to a complete-data problem. 
Here we will study the general rank problem and leave the rank one case for Section \ref{subsec:FIM1}.

\subsubsection{Missing information principle and Oakes' theorem}\label{sec:missinginfo}

Unlike second-order optimization algorithms such as Newton-Raphson, the Hessian of the loglikelihood is not necessary in the EM algorithm. 
However, this Hessian is important to obtain the observed and expected Fisher information matrices. 
We obtain the Hessian through the expected complete loglikelihood $Q(\btheta,\bar\btheta)$ of Equation \eqref{eq:Qfunc2}, which avoids the difficulty in
directly differentiating the loglikelihood.
The negative Hessian of the loglikelihood decomposition \eqref{eq:llik_completemiss2} and the definition of observed Fisher information enable us to define 
\begin{align}
\FIM_{obs}(\btheta,\bar\btheta,\bx)& \coloneqq -
\dfrac{\partial^2}{\partial \btheta\, \partial \btheta^\top} 
Q(\btheta,\bar\btheta)
+ \dfrac{\partial^2}{\partial \btheta\, \partial \btheta^\top} 
H(\btheta,\bar\btheta) \,. \nonumber
\intertext{Note that $\FIM_{obs}(\btheta,\bar\btheta,\bx) = \FIM_{obs}(\btheta,\bx)$  holds for any value of $\bar\btheta$. If we choose $\bar\btheta=\btheta$ then we have}
\FIM_{obs}(\btheta,\bar\btheta,\bx)\Big|_{\bar\btheta=\btheta} &  =
\underbrace{
\left[-\dfrac{\partial^2}{\partial \btheta\, \partial \btheta^\top} 
Q(\btheta,\bar\btheta)\right]_{\bar\btheta = \btheta}}_{ \let\scriptstyle\textstyle
\substack{\FIM_c(\btheta,\bx)}}
-
\underbrace{
\left[-\dfrac{\partial^2}{\partial \btheta\, \partial \btheta^\top} 
H(\btheta,\bar\btheta)\right]_{\bar\btheta = \btheta}}_{ \let\scriptstyle\textstyle
\substack{\FIM_m(\btheta,\bx)}} \nonumber \\
& = \FIM_{obs}(\btheta,\bx) \label{eq:missinginfoprinc}
\end{align}
where 
$
\FIM_c(\btheta,\bx)
$
is the complete Fisher information matrix, and 
$
\FIM_m(\btheta,\bx)
$
is the missing Fisher information matrix. Equation \eqref{eq:missinginfoprinc} has the interpretation that the information observed $\FIM(\btheta,\bx)$ is the complete information $\FIM_c(\btheta,\bx)$ minus the information missing from the latent process $\FIM_m(\btheta,\bx)$. This is commonly known as the missing information principle \citep{orchardandwoodbury72}. 
In our PCP example, the missing information principle states that the process of going from the unobserved latent tensor $\ten{Z}$ to the observed tensor $\ten{X}$ leads to a loss of quantifiable information as a result of the latent mechanism $\ten{X} = \ten{Z}\bar\times_1\bone$ (see the notation section \ref{sec:notation}).

Although equation \eqref{eq:missinginfoprinc} states the observed Fisher information $\FIM(\btheta,\bx)$ in terms of the complete Fisher information $\FIM_c(\btheta,\bx)$ which is usually easier to find, it also involves the missing Fisher information $\FIM_m(\btheta,\bx)$ that can be as complex as $\FIM(\btheta,\bx)$. Many techniques have been proposed to express $\FIM_m(\btheta,\bx)$ analytically or numerically in terms of the complete loglikelihood only \citep{mclachlanandthri07,louis82}. Here we will use the method of \citet{oakes99} because, unlike the popular method of \citet{louis82}'s, it allows us to obtain the observed Fisher information at any parameter value. We summarize \citet{oakes99}'s result in the following theorem.
\begin{theorem}[\citep{oakes99}]\label{thm:oakesHess}
The missing Fisher information $\FIM_m(\btheta,\bx)$ of Equation \eqref{eq:missinginfoprinc} can be written in terms of derivatives of the expected complete loglikelihood: 
$$
\FIM_m(\btheta,\bx) = \left[
\dfrac{\partial^2}{\partial \btheta\, \partial \bar\btheta^\top} Q(\btheta,\bar\btheta)
\right]_{\bar\btheta = \btheta}
$$
when the interchanging of integration with respect to $\bz$ and differentiation in $\btheta$ holds for $\log p(\bz|\bx)$.
\end{theorem}
This is the key result that allows us to derive the Fisher information for the rank greater than one case.
We note that the interaction of integration with respect to $\bz$ and differentiation in Theorem \ref{thm:oakesHess} is satisfied by the Leibniz integral rule whenever the sample space of $\bz|\bx$ is not a function of $\btheta$, as is the case in our formulation.  

\subsubsection{Fisher information matrix}
Let
$\ten{M} = [\![\mat{A}_1,\dots,\mat{A}_N]\!]$ and recall the notation of Section \ref{sec:notation}. We now use the missing information principle and Oakes' theorem to calculate in close form the Fisher information matrix for the PCP model.
\begin{theorem}\label{thm:OFIM_genR}
The observed $\FIM_{obs}(\btheta,\bx)\in\mbbR^{(R\sum_qN_q)\times (R\sum_qN_q)}$ and expected $\FIM(\btheta)\in\mbbR^{(R\sum_qN_q)\times (R\sum_qN_q)}$ Fisher information matrices 
of the PCP model are $P\times P$ block matrices, where each $(k,l)$ block (of size $RN_k\times RN_l$, $k,l\in [P]$) is itself an $R\times R$ block matrix with $(r,s)$ sub-block of size $N_k\times N_l$, $r,s\in [R]$. That is,
$$
\FIM_{obs}(\btheta,\ten{X})
= 
\left\{
\left\{
\mat{D}_{k,l}^{r,s}\big(\ten{X}\oslash \ten{M}^{*2}\big) + \mat{F}_{k,l}^{r,s}
\right\}_{r,s}
\right\}_{k,l}
\quad \text{and} \quad 
\FIM(\btheta) 
= 
\left\{
\left\{
\mat{D}_{k,l}^{r,s}(\ten{M}^{*-1})
\right\}_{r,s}
\right\}_{k,l}.
$$
The matrices 
$\mat{D}_{k,l}^{r,s}(\ten{Y}),\mat{F}_{k,l}^{r,s}\in \mathbb{R}^{N_k\times N_l}$ are (here $\ba_{k,r}$ denotes the $r$-th column of $\mat{A}_k$):
$$
\mat{D}_{k,l}^{r,s}(\ten{Y}) = 
\begin{cases}
\diag\Big(
\ten{Y}\bar\times_{q\neq k} (\ba_{q,r}*\ba_{q,s})
\Big)& k = l\\
\left(\ba_{k,s}\ba_{l,r}^\top \right)*
\mat{Y}_{(k)} & k \neq l \text{ and } P=2\\
\left(\ba_{k,s}\ba_{l,r}^\top \right)*
\big(
\ten{Y}\bar\times_{q\neq k,l} (\ba_{q,r}*\ba_{q,s})
\big) & k < l \text{ and } P>2
\end{cases},
$$
$$
\mat{F}_{k,l}^{r,s} = 
\begin{cases}
0 &k = l \text{ or } r\neq s\\
\bone_{N_k}\bone_{N_l}^\top - 
(\mat{X}_{(k)}\oslash\mat{M}_{(k)}) 
& k \neq l \text{ and } r = s \text{ and } P=2 \\
\blambda_{-k,l}[r]\bone_{N_k}\bone_{N_l}^\top - 
\big(
(\ten{X}\oslash\ten{M})\bar\times_{q\neq k,l} (\ba_{q,r})
\big)
& k < l \text{ and } r = s \text{ and } P>2 
\end{cases}.
$$
When $P>2$ and $k>l$ we have $\mat{D}_{k,l}^{r,s}(\ten{Y}) = {\mat{D}_{l,k}^{r,s}(\ten{Y})}^\top$  and $\mat{F}_{k,l}^{r,s} = {\mat{F}_{l,k}^{r,s}}^\top$.
\end{theorem}
\begin{proof}
See Appendix \ref{proof:Hessian}. 
\end{proof}

The observed or expected Fisher information matrices can be computed efficiently based on the following simplifications. 
First, the matrix $((\odot\mat{A}_{[P]\setminus\{k\}})^\top \odot\mat{Y}_{(k)} )\odot\mat{A}_{[P]\setminus\{k\}}$ contains all the non-zero entries of $\mat{D}_{k,k}^{r,s}(\ten{Y})$, across all $r,s$. 
When $P>2$ and $k\neq l$, the matrix
$
\big((\odot\mat{A}_{[P]\setminus\{k, l\}})^\top (\odot\mat{A}_{[P]\setminus\{k, l\}})^\top)\mat{Y}_{(k,l)}^\top\big)
*(\mat{A}_k^\top \otimes \mat{A}_l^\top)
$
contains the entries of $\mat{D}_{k,l}^{r,s}(\ten{Y})$, across all $r,s$. Here  $\mat{Y}_{(k,l)}$ is a matricization of $\ten{Y}$ that brings the $(k,l)$ modes to the rows, and the remainder modes to the columns.

In Section \ref{sec:MCFIM} we will use Monte-Carlo simulations to corroborate the validity of our expression for $\FIM(\btheta)$ in Theorem \ref{thm:OFIM_genR}. In these experiments we will use Bartlett's moment identities \citep[sec. 2]{bartlett53}, which state that when the interchanging of integration with respect to $\bx$ and differentiation in $\btheta$ holds for the loglikelihood $\ell(\bx|\btheta)$, then 
\begin{equation}\label{eq:bartletts}
\mbbE\Big( \nabla_{\btheta} \ell(\bx|\btheta)\Big) =\bzero
\quad \text{and} \quad
\mbbE\Big( (\nabla_{\btheta} \ell(\bx|\btheta))(\nabla_{\btheta} \ell(\bx|\btheta))^\top\Big) = \FIM(\btheta),
\end{equation}
where $\nabla_{\btheta} \ell(\bx|\btheta)$ is the score function.  For the PCP model with $\ell(\bx|\btheta)$ in \eqref{eq:llik1},  the score function is
\begin{equation}\label{eq:score1}
\nabla_{\btheta} \ell(\bx|\btheta)  = \begin{bmatrix}
\big(\nabla_{\vecc(\mat{A}_1)} \ell(\bx|\btheta)\big)^\top & \cdots & \big(\nabla_{\vecc(\mat{A}_P)} \ell(\bx|\btheta)\big)^\top
\end{bmatrix}^\top,
\end{equation}
where 
$
\nabla_{\vecc(\mat{A}_p)} \ell(\bx|\btheta) = \vecc\Big((\mat{X}_{(p)} \oslash \mat{M}_{(p)} - 1)_{(p)}(\odot\mat{A}_{[P]\setminus\{p\}})\Big).
$
In Theorem \ref{thm:OFIM_genR} we chose to apply Oakes' theorem and not Bartlett's identities, because Oakes' theorem also enables us to obtain the observed Fisher information matrix $\FIM_{obs}(\btheta,\bx)$.

\subsubsection{Fisher information matrix rank}

The rank of the expected Fisher information matrix $\FIM(\btheta)$ reveals the dimensionality of the parameter space over which parameter inference is made. Hence, the rank provides insight into model complexity; a higher rank indicates a more complex model with greater flexibility in fitting the data, whereas a lower rank signifies a more parsimonious model. The following conjecture states the rank of the Fisher information matrix $\FIM(\btheta)$ of Theorem \ref{thm:OFIM_genR}. We prove it for the special case of $R=1$ in Theorem \ref{thm:Rank1_identf}, and provide numerical evidence for the case of rank greater than one in Section \ref{sec:sim_identif}.

\begin{conjecture}\label{conjecture:rank}
    Consider the Fisher information matrix $\FIM(\btheta)$ of Theorem \ref{thm:OFIM_genR} (of size $(R\sum_qN_q)\times(R\sum_qN_q)$). Let $\btheta$ be such that its components are strictly positive, and no two distinct subsets of the components are linearly dependent. Then
    $$
    \text{rank}(\FIM(\btheta)) = \min\Big(R\sum_qN_q-L, \prod_qN_q\Big),
    $$
where $L=(\min\{R,N_1,N_2\})^2$ if $P=2$, and $L=R(P-1)$ if $P>2$.
\end{conjecture}
The condition that no two distinct subsets of the components of $\btheta$ are linearly dependent ensures that the components do not introduce redundancy among themselves. From our conjecture we conclude that $\text{rank}(\FIM(\btheta))<R\sum_qN_q$, indicating that $\FIM(\btheta)$ is always singular. This singularity can be categorized into two cases. 
The first case occurs when $\prod_qN_q<R\sum_qN_q-L$, where the dimensionality of the parameter space exceeds that of the sample space; thus, either the rank $R$ should be reduced, or the dimensions $N_1,\dots,N_P$ should be increased. 
The second case arises when $\prod_qN_q>R\sum_qN_q-L$. Here, the null space of $\FIM(\btheta)$ has dimension $L$.
When $P>2$, $L = R(P-1)$ because there are $R$ distinct rank one components of $\ten{M}$ having $P-1$ redundancies each, as per the unidentifiability of CP decompositions discussed in Section \ref{sec:relatedwork}. 
When $P=2$, we have that $\mat{M} = \mat{A}_1\mat{A}_2^\top$ is a low-rank matrix. In this case, the dimension of the null space of $\FIM(\btheta)$ is $(\min\{R,N_1,N_2\})^2$ because of possible rotations in factor columns.

Having studied the $R>1$  case through a latent-variable formulation, we will now study in more detail the $R=$1 case. While the $R=$1 case can be framed as a special case of the results we found in this section, it also has important simplifications that make inference more straight forward.

\section{Rank One PCP Inference}\label{sec:rank1} 
Parameter inference for the rank one case simplifies considerably because the latent variable $\ten{Z}$ equals the observed random variable $\ten{X}$, the factor matrices are vectors that we denote by $\ba_1\dots,\ba_P$, the parameter vector is $\btheta = \begin{bmatrix}\ba_1^\top& \cdots & \ba_P^\top\end{bmatrix}^\top \in \mathbf{R}^{N_1+N_2+ \cdots + N_p}$, and the PCP model is written as
\begin{equation}\label{eq:pcpR1}
    x_{\bi}\sim\Poi\left(\prod_{p=1}^P \ba_{p}(i_p)\right),
\end{equation}
where $\ba_{p}(i_p)$ denotes entry $i_p$ in vector $\ba_{p}$.
Recall from the introduction that we will assume that every entry of $\btheta$ is positive; otherwise the Poisson distribution of Equation \eqref{eq:pcpR1} is undefined.

\subsection{Rank one latent variable formulation}\label{sec:r1_lvf}

Because the complete data vector $\bz$ is equal to the observed data vector $\bx$ in the rank one case, the conditional random variable $\bz|\bx$ is degenerate and therefore the missing loglikelihood $\ell_m(\bz|\bx,\btheta)$ is zero.
In this case, the loglikelihood decomposition \eqref{eq:llik_completemiss} holds with $\ell(\bx|\btheta)=\ell_c(\bz|\btheta)$ and simplifies to
\begin{subequations}
\begin{equation}\label{eq:rank1loglik}
\begin{aligned}
        \ell(\bx|\btheta) &= 
        \sum_{\bi}\left[x_{\bi}\log\left(\prod_{p=1}^P \ba_{p}(i_p)\right)-\left(\prod_{p=1}^P \ba_{p}(i_p)\right)\right]+C
        \\&=
        \sum_{p=1}^P\left[
        \bx_p^\top\log(\ba_p)
        \right]
- \lambda  +C,
\end{aligned}
\end{equation}
where $\bx_p :=\mat{X}_{(p)}\bone_{N/N_p} \in \mathbb{N}_{0}^{N_p} $, the constant $C$ does not depend upon $\btheta$, and 
\begin{align}
\lambda \coloneqq \prod_{p=1}^P \lambda_p \text{ and } \lambda_p \coloneqq \bone_{N_p}^\top \ba_p.
\end{align}
\end{subequations}
The conditional expectation decomposition~\eqref{eq:llik_completemiss2} simplifies to $\ell(\bx,\btheta) = Q(\btheta,\bar\btheta)$ because $H(\btheta,\bar\btheta) = 0$. 
Furthermore, $\mat{\bar Z}_p$ matrix of Equation \eqref{eq:Zmat} simplifies (for any $\bar \btheta$) to
\begin{equation*}
    \begin{aligned}
        \mat{\bar Z}_p &=
 \bar \ba_p * \left([\ten{X}_{(p)}\oslash ( \bar \ba_p (\odot {\bar \ba}_{[P]\setminus\{p\}})^\top)] \odot {\bar \ba}_{[P]\setminus\{p\}}\right)
 =  
\mat{X}_{(p)} \bone_{N/N_p} = \bx_p.
    \end{aligned} 
\end{equation*}
We remark that Theorem~\ref{thm:Qfunc} implies that apart from a scaling $\mat{\bar Z}_p$ maximizes $\ell(\bx,\btheta) = Q(\btheta,\bar\btheta)$. 
Moreover, as we will show in Lemma~\ref{lem:solve-mle}, the rank one estimator for the model $\ten{M}$ is unique.
This means that the row-sum vector $\bx_p$ is the special case of the matrix $\mat{\bar Z}_p$ in Equation \eqref{eq:Zmat}. Hence, according to Corollary~\ref{cor:Qfunc}, if $\lambda_p=\lambda$ then the loglikelihood of $\bx_p$ is that of $\bx_p\sim \Poi(\ba_p)$ up to a constant. This idea is critical for framing the rank one case as an instance of the case of rank greater than one that we studied in Section \ref{sec:genrank}. 
In the general case, inference on each factor $\mat{A}_p$ was done using $Q(\btheta,\bar\btheta)$ through $\mat{\bar Z}_p$. In the rank one case, inference for each factor $\ba_p$ will be done using $\ell(\bx|\btheta)$ through $\bx_p$. 
Inference based on $\bx_p$ is convenient because the objective function $\ell(\bx|\btheta)$ is concave on $\ba_p$, and it depends only on $\bx_p$.

We denote the gradient and Hessian of the loglikelihood by $\nabla_{\btheta} \ell(\bx|\btheta)$ and $\nabla^2_{\btheta}\ell(\bx|\btheta)$, where $
\btheta =\begin{bmatrix}\ba_1^\top& \cdots & \ba_P^\top\end{bmatrix}^\top$. We have
\begin{align*}
\nabla_{\btheta} \ell(\bx|\btheta) = \begin{bmatrix}
\big(\nabla_{\ba_1} \ell(\bx|\btheta)\big)^\top & \cdots & \big(\nabla_{\ba_P} \ell(\bx|\btheta)\big)^\top
\end{bmatrix}^\top \in \mathbb{R}^{N_1+N_2+\cdots+N_P}
\end{align*}
and so it suffices to consider
\begin{subequations} \label{deriv-loglike}
\begin{align}\label{eq:gradientR1}
 \nabla_{\ba_p} \ell(\bx|\btheta)& = \bx_p \oslash \ba_p
        - \Big(\prod_{q\neq p } \lambda_q\Big) \bone_{N_p} \in \mathbb{R}^{N_p}.
\end{align}
Similarly, $\nabla^2_{\btheta}\ell(\bx|\btheta)$ is a block matrix with $(p,q)$-block
\begin{equation}\label{eq:hessian}
    \nabla^2_{\ba_p,\ba_q}\ell(\bx|\btheta)
    =\begin{cases}
        -\diag(\bx_p\oslash\ba_p^{*2}) \in \mathbb{R}^{N_p \times N_p}&\text{if } p=q \\
        -\lambda\lambda_{pq}^{-1}\bone_{N_p}\bone_{N_q}^\top \in \mathbb{R}^{N_p \times N_q} &\text{otherwise} 
    \end{cases}
\end{equation}
\end{subequations}
where $\lambda_{pq} = \lambda_p\lambda_q$ and $\ba_p^{*2}\coloneqq \ba_p*\ba_p$ denotes a Hadamard power. 

\subsection{Rank one maximum likelihood estimation}\label{sec:MLErank1}

Our interest is in estimating a rank one $P$-way tensor 
\begin{align}
\ten{M} =  \ba_1 \circ \cdots \circ \ba_P \in \mathbb{R}_{+}^{N_1 \times N_2 \times \cdots \times N_p}
\end{align}
for $\ten{X}$, which maximizes the loglikelihood in Equation \eqref{eq:rank1loglik}.
A key quantity is the sum of all the elements of the tensor $\ten{X}$, which is given by $\bone_{N}^\top \vecc(\ten{X})$. This sum is positive for a nonzero tensor.
The following lemma explains that a constraint involving $\bone_{N}^\top \vecc(\ten{X})$ and $\btheta$, that is, the factor vectors $\ba_1\dots,\ba_P$, produces a well-posed optimization problem with a closed-form solution. This is somewhat remarkable because the critical point is given by the solution of a constrained nonlinear set of equations. 
\begin{lemma} \label{lem:solve-mle}
Let $\eqref{eq:rank1loglik}$ denote the loglikelihood for a rank one PCP model with gradient \eqref{eq:gradientR1}. 
If $\bx_p=\mat{X}_{(p)}  \bone_{N/N_p}$ is a positive vector for $ p \in [P]$ then 
\begin{subequations}
\begin{align}
&\max_{\btheta \in \mathbb{R}^{N_1 + N_2 + \cdots + N_p,\,+} } \ell(\btheta) \\
\intertext{ subject to the multi-linear constraint} 
&\lambda:=\lambda_1 \cdots \lambda_P =  \bone_{N}^\top \vecc(\ten{X}) > 0 \label{lem:constraint}
\end{align}
\end{subequations}
has the solution
\begin{align}\label{r1mM}
\hmM = \hat\ba_1 \circ \cdots \circ \hat\ba_P & =
\frac{1}{\lambda^{P-1}} \, \bx_1 \circ \cdots \circ \bx_P\;.
\end{align}
\end{lemma}
\begin{proof}
\begin{subequations}
The gradient \eqref{eq:gradientR1} implies that
the critical point is given by positive vector
\begin{align} \label{lem:1}
\hat\ba_p & = \frac{1}{\prod_{q\neq p }^P \lambda_q } \mat{X}_{(p)}  \bone_{N/N_p} \in \mathbb{R}_{+}^{N_p} \text{ for } p=1,\ldots, P\,.
\end{align}
If we take the dot product of both sides with $\bone_{N_p}$ then the constraint \eqref{lem:constraint} follows.
Because $\lambda_1 > 0, \ldots, \lambda_P > 0$ if and only if $ \bone_{N_p}^\top \mat{X}_{(p)}  \bone_{N/N_p}> 0$ we can rewrite \eqref{lem:1} as
\begin{align} \label{lem:2}
\hat\ba_p & =  \frac{\lambda_p}{\lambda_1 \cdots \lambda_P} \mat{X}_{(p)}  \bone_{N/N_p} \in \mathbb{R}_{+}^{N_p}\,.
\end{align}
\end{subequations} 
The rank one tensor $\hmM$ in \eqref{r1mM} now follows by forming the order $P$ rank one tensor from \eqref{lem:2} for $p=1,\ldots,P$ and each element of $\hmM$ is positive because each factor vector is positive.
\end{proof}

An important consequence of Lemma \ref{lem:solve-mle} is that the rank one tensor $\hmM$ is unique and can be expressed in terms of the nonzero data tensor $\ten{X}$ while the factor vectors are unique up to the scaling in \eqref{lem:1}. 
The constraint \eqref{lem:constraint} explains that the product of the $P$ scalars $\lambda_p$ must be equal to the sum of the elements of $\ten{X}$. 
For example, two choices are to set $\lambda_p^{1/P}= \bone_{N}^\top \vecc(\ten{X})$ for all $p$, or to set one $\lambda_p= \bone_{N}^\top \vecc(\ten{X})$ and all the other $\lambda_q$  to a simplex $(q\neq p)$.

In the following lemma we show that while the factor vectors $\ba_1,\dots,\ba_P$ are non-unique, choosing a specific constraint set can help us estimate them uniquely, and find properties such as bias.

\begin{lemma}\label{lemma:rank1bias}
Consider the constraint set $\tilde\Theta$ that assigns all the weight to the first factor vector $\ba_1$ by constraining the remaining factor vectors to a simplex, i.e,
\begin{equation}\label{eq:Theta}
\tilde\Theta = \Big\{\btheta = \begin{bmatrix} \ba_1^\top & \cdots& \ba_P^\top\end{bmatrix}^\top: \bone_{N_2}^\top\ba_2 \!=\! \dots \!=\! \bone_{N_P}^\top\ba_P \!=\! 1\Big\}.
\end{equation}
Any $\btheta=\begin{bmatrix} \ba_1^\top & \cdots& \ba_P^\top\end{bmatrix}^\top$ can be parameterized to belong to $\tilde\Theta$, since
\begin{equation}\label{eq:identifR1}
    \tilde\btheta = \begin{bmatrix} \tilde\ba_1^\top & \cdots& \tilde\ba_P^\top\end{bmatrix}^\top:= 
    \begin{bmatrix}
    \ba_1^\top\lambda\lambda_1^{-1} & \ba_2^\top\lambda_2^{-1} &\dots&\ba_P^\top\lambda_P^{-1}        
    \end{bmatrix}^\top\in\tilde\Theta.
\end{equation}
Let $\lambda = \bone_N^\top\vecc(\ten{X})$. Then the unique MLE over $\tilde\Theta$ is 
\begin{equation*}\label{eq:mle_r1}
    \hat\btheta :=\begin{bmatrix}
    \hat\ba_1\\ \hat\ba_2\\ \vdots\\ \hat \ba_P
\end{bmatrix} = \begin{bmatrix}
    \mat{X}_{(1)}\bone_{N/N_{1}}\\\lambda^{-1}\mat{X}_{(2)}\bone_{N/N_{2}}\\
    \vdots\\    \lambda^{-1}\mat{X}_{(P)}\bone_{N/N_{P}}
\end{bmatrix}.
\end{equation*}
Furthermore, $\hat\btheta$ is an unbiased estimator of $\tilde\btheta$ of Equation \eqref{eq:identifR1}. That is,
$\mbbE(\hat\btheta) = \tilde\btheta$.
\end{lemma}
\begin{proof}

If
$\hat\btheta:=
\begin{bmatrix} 
\hat\ba_1^\top & \cdots& \hat\ba_P^\top\end{bmatrix}^\top
\in\tilde\Theta$, then $\prod_{q\neq 1}(\bone_{N_q}^\top\hat\ba_q) = 1$, and  Equation \eqref{lem:1} simplifies to $\hat\ba_1 = \mat{X}_{(1)}\bone_{N/N_{1}}$.
For all other $p\in\{2,3,\dots,P\}$ it holds that $\prod_{q\neq p}(\bone_{N_q}^\top\hat\ba_q) = \bone_{N_1}^\top\hat\ba_1 = \lambda$, and Equation \eqref{lem:1} simplifies to $\hat\ba_p = \lambda^{-1} \mat{X}_{(p)}\bone_{N/N_{p}}$.
This establishes that $\hat\btheta$ is the unique critical point of the score function that is contained in the constraint set \eqref{eq:Theta}. To show it's an unbiased estimator,  note that under the rank one PCP model, for any $p$ it holds that $\mbbE(\mat{X}_{(p)}) = \ba_p(\otimes_{q\neq p}\ba_q)^\top$. Hence, using the notation of Equation \eqref{eq:identifR1}, we have that 
\begin{equation*}
    \begin{aligned}
\mbbE(\hat\ba_1) = \mbbE(\mat{X}_{(1)})\bone_{N/N_1} 
&= \ba_1(\otimes_{q\neq 1}\ba_q)^\top\bone_{N/N_1} 
\\&= \ba_1(\otimes_{q\neq 1}\ba_q^\top\bone_{N_q})
\\&= \ba_1(\lambda\lambda_1^{-1}) = \tilde \ba_1.
    \end{aligned}
\end{equation*}
For the other $\mbbE(\hat\ba_p)$ we will use the fact that if $X$ and $Y$ are independent Poisson random variables, then 
$\mbbE\big(X/(X+Y)\big) =\mbbE(X) / \mbbE(X+Y)$ holds.
Using this identity element-wise for each entry of $\mat{X}_{(p)}\bone_{N/N_{p}}$, we have that for any $p\in\{2,3,\dots,P\}$
\begin{equation*}
    \begin{aligned}
\mbbE(\hat\ba_p) = \mbbE\left(\dfrac{\mat{X}_{(p)}\bone_{N/N_{p}}}{\bone_{N_p}^\top \mat{X}_{(p)}  \bone_{N/N_p}}\right)
&= \dfrac{\mbbE(\mat{X}_{(p)}\bone_{N/N_{p}})}{\mbbE(\bone_{N_p}^\top \mat{X}_{(p)}  \bone_{N/N_p})}
\\&= \dfrac{\lambda\lambda_p^{-1}\ba_p}{\lambda} = \tilde\ba_p.
    \end{aligned}
\end{equation*}
Hence, it holds that for any $p$, $\mbbE(\hat\ba_p) = \tilde\ba_p$, and hence $\mbbE(\hat\btheta) = \tilde\btheta$.
\end{proof}

Having found maximum likelihood estimators for this rank one case, and framed it as an instance of the general case with important simplifications, now we proceed with obtaining the Fisher information matrix.

\subsection{CP rank one Fisher information matrix}\label{subsec:FIM1}

We discussed that in the rank one case, the complete data vector $\bz$ is equal to the observed data vector $\bx$. An important consequence of this fact can be viewed with the lens of the missing information principle of Equation \eqref{eq:missinginfoprinc}. Here, since there is nothing missing in $\bz$ (that is, it is observed), the missing Fisher information matrix $\FIM_m(\btheta,\bx)$ is zero, and the observed Fisher information matrix $\FIM_{obs}(\btheta,\bx)$ equals the complete Fisher information matrix $\FIM_{c}(\btheta,\bx)$. This result is intuitive, but also comes from the fact that $H(\btheta,\bar\btheta) = 0$, which we established in Section \ref{sec:r1_lvf}. The fact that $\FIM_m(\btheta,\bx) = 0$ is convenient because we do not need to use Oakes' theorem to find it, and instead, we can take derivatives directly to the loglikelihood. In the following lemma we obtain the observed and expected Fisher information matrices directly through Equations \eqref{eq:gradientR1} and \eqref{eq:hessian}.

\begin{lemma}\label{thm:FIM_rank1}
Let $\eqref{eq:rank1loglik}$ denote the loglikelihood for a rank one PCP model with Hessian \eqref{eq:hessian}. 
If $\bx_p:=\mat{X}_{(p)}  \bone_{N/N_p}$ is a positive vector for $ p \in [P]$ then 
the observed Fisher information matrix $\FIM_{obs}(\btheta,\bx)$ for the rank one PCP model is
\begin{subequations}
\begin{align}
\FIM_{obs}(\btheta,\bx)  = 
\begin{bmatrix}
\diag(\bx_1\oslash\ba_1^{*2}) &\lambda\lambda_{12}^{-1}\bone_{N_1}\bone_{N_2}^\top &\dots&\lambda\lambda_{1P}^{-1}\bone_{N_1}\bone_{N_p}^\top \\
\lambda\lambda_{12}^{-1}\bone_{N_2}\bone_{N_1}^\top &\diag(\bx_2\oslash\ba_2^{*2}) & \dots&\lambda\lambda_{2P}^{-1}\bone_{N_2}\bone_{N_P}^\top \\
\vdots & \vdots & \ddots & \vdots\\
\lambda\lambda_{1P}^{-1}\bone_{N_P}\bone_{N_1}^\top & \lambda\lambda_{P2}^{-1}\bone_{N_P}\bone_{N_2}^\top & \dots & \diag(\bx_P\oslash\ba_P^{*2}) \label{eq:FIMR1_obs}
\end{bmatrix}\,,
\end{align}
and the expected Fisher information matrix $\FIM(\btheta)$ of the rank one PCP model is
\begin{align}
\FIM(\btheta)  = 
\lambda\begin{bmatrix}
\lambda_{1}^{-1} \diag(\ba_1)^{-1} &\lambda_{12}^{-1}\bone_{N_1}\bone_{N_2}^\top &\dots& \lambda_{1P}^{-1}\bone_{N_1}\bone_{N_P}^\top \\
\lambda_{12}^{-1}\bone_{N_2}\bone_{N_1}^\top &\lambda_{2}^{-1}\diag(\ba_2)^{-1}  & \dots& \lambda_{2P}^{-1}\bone_{N_2}\bone_{N_P}^\top \\
\vdots & \vdots & \ddots & \vdots\\
\lambda_{1P}^{-1}\bone_{N_P}\bone_{N_1}^\top & \lambda_{P2}^{-1}\bone_{N_P}\bone_{N_2}^\top & \dots & \lambda_{P}^{-1}\diag(\ba_P)^{-1}  \label{eq:FIMR1_exp}
\end{bmatrix}.
\end{align}
\end{subequations}
\end{lemma}
\begin{proof}
Since we are assuming that the factor vectors $\ba_p$ are positive, the observed Fisher information matrix $\FIM_{obs}(\btheta,\bx)$ is given by the negative of the Hessian \eqref{eq:hessian}.  
The expression for the expected Fisher information matrix follows from the definition $\FIM(\btheta) = \mbbE\big(\FIM_{obs}(\btheta,\bx)\big)$, and
$   \mbbE(\bx_p) =
    \ba_p\lambda\lambda_{p}^{-1}
$ (see the proof of Lemma \ref{lemma:rank1bias}). Hence, we have
\begin{align*}
    \mbbE(\bx_p\oslash\ba_p^{*2}) = \mbbE(\bx_p) \oslash\ba_p^{*2} = \lambda\lambda_{p}^{-1}\bone_{N_p} \oslash \ba_p\,.
\end{align*}
which completes the proof after noting that $\diag(\bone_{N_p} \oslash \ba_p) = \diag(\ba_p)^{-1}$.
\end{proof}

Lemma  \ref{thm:FIM_rank1} provides Fisher information matrices for any factor scaling, as long as $\lambda = \prod_p (\ba_p^\top\bone_{N_p})$. Simplifications can be made for specific scalings. For example, if $\tilde\btheta = [\tilde\ba_1^\top\dots\tilde\ba_P^\top]^\top$ has scaled factors so that $\bone_{N_p}^\top\tilde\ba_p = \lambda^{1/P}$ for all $p$, then 
$$
\FIM(\tilde\btheta)  = 
\lambda^{1-2/P}
\begin{bmatrix}
\lambda^{1/P}\diag(\tilde\ba_1)^{-1} & \dots&  \bone_{N_1}\bone_{N_P}^\top \\
\vdots & \ddots & \vdots\\
\bone_{N_P} \bone_{N_1}^\top & \dots & \lambda^{1/P}\diag( \tilde\ba_P)^{-1} \label{fim}
\end{bmatrix},
$$
The matrix above can be split it into a block-diagonal matrix plus a rank one matrix. This fact will be leveraged later in Theorem \ref{thm:Rank1_identf} to  find many of its spectral properties. Note that if $\btheta$ has an arbitrary scaling, then there is a non-singular diagonal matrix $\Gamma$ that satisfies $\tilde\btheta = \Gamma\btheta$, ($\Gamma$ here is block-diagonal with $p$th diagonal block $\lambda^{1/P}\lambda_p^{-1}I_{N_p}$), and hence the Fisher information matrix $\FIM(\btheta)$ of Equation \eqref{eq:FIMR1_exp} can be written in terms of that of the simplified Fisher information matrix as 
\begin{equation}\label{eq:FIMequivalency}
    \FIM(\btheta) = \Gamma \, \FIM(\tilde\btheta) \, \Gamma^\top.
\end{equation}
The identity above implies that while different scalings will lead to different Fisher information matrices, all proper scalings (proper meaning that $\Gamma$ is non-singular) will satisfy \eqref{eq:FIMequivalency}, and hence, will result in equivalent Fisher information matrices up to some scaling $\Gamma$. 
We can now establish conditions for identifiability using these two Fisher information matrices.

\subsubsection{Rank one identifiability}\label{sec:identifR1}

Identifiability in statistical modeling ensures that different parameter values lead to different probability distributions of the observed data, allowing for reliable inference. In the context of the rank one PCP model, we say that PCP is non-identifiable if there exists $\btheta_1\neq\btheta_2$ for which $\ell(\btheta_1)=\ell(\btheta_2)$ and identifiable otherwise. 

The discussion following Lemma~\ref{lem:solve-mle} explains that the rank one PCP model $\ten{M}$ is unique independent of the scaling ambiguity with the factor vectors. The ambiguity arises because numerous scalings satisfy the constraint \eqref{lem:constraint}. Because the factor vectors comprise the parameter vector $\btheta$, the rank one PCP model is not identifiable and suggests that $P-1$ components of $\btheta$ are redundant. 
Such a redundancy is not unique to the Poisson distribution and has been known for general CP models \citep{sidiropoulosandbro00}, but the over-parameterization of $\btheta$ has not been studied through the lenses of the Fisher information matrix. 
The benefit of an information-theoretic approach is the ability to recast the analysis of low-rank tensor models using the tools of mathematical statistics. For instance, can we precisely describe in what sense the inference problem is well-posed and the impact upon the determination of an efficient estimator?

The expected and observed Fisher information matrices obtained in Theorem \ref{thm:FIM_rank1} play crucial roles in assessing identifiability. The expected Fisher information matrix $\FIM(\btheta)$ provides a measure of the average information content about the parameters across all possible data sets and is used to assess structural identifiability. If $\FIM(\btheta)$ is non-singular, it suggests that the parameters are structurally identifiable, meaning the model structure allows for unique parameter estimation. On the other hand, the observed Fisher information matrix $\FIM_{obs}(\btheta,\bx)$ provides information about a sample-specific measure of the information content about the parameters. If $\FIM_{obs}(\btheta,\bx)$ is non-singular, it indicates that the parameters are identifiable given the specific instance $\ten{X}$. In the following theorem we demonstrate that $\FIM(\btheta)$ is singular and determine its rank. 
In the discussion that follows the theorem, we relate the singularity of $\FIM(\btheta)$ to the constraint \eqref{lem:constraint}.  

\begin{theorem}\label{thm:Rank1_identf}
If 
$\btheta_{\bullet} = \begin{bmatrix}
    \ba_1^\top&\ba_{2\bullet}^\top&\dots&\ba_{P\bullet}^\top
\end{bmatrix}^\top\in\mbbR_{+}^r$, where $\ba_p^\top =\begin{bmatrix}
    a_{p1} & \ba_{p\bullet}^\top
\end{bmatrix}^\top \in\mbbR_{+}^{N_p} $ for $p \in \lbrace 2,\ldots, P \rbrace$ and $r = \big(\sum_{p=1}^P N_p \big)-P+1$ then 
\begin{enumerate}
    \item The matrix $
\FIM_F(\btheta) 
\coloneqq 
-\mbbE\left(\nabla^2_{\btheta_{\bullet},\btheta_{\bullet}}  \ell(\btheta)\right)
$
is non-singular of rank $r$.
\item 
The Fisher information matrix $\FIM(\btheta)$ is singular of rank $r$.
\item
The rows of 
$\mathbf{H}^\top \coloneqq \begin{bmatrix}
        -\FIM_{F,N}(\btheta)\FIM_F^{-1}(\btheta)  &
        \mathbf{I}_{P-1}
\end{bmatrix}^\top \in \mathbb{R}^{(P-1) \times (r+P-1)}$ are basis for the nullspace of $\FIM(\btheta)$.
\end{enumerate}
\end{theorem}
\begin{proof}
Without loss of generality permute the rows and columns of $\FIM(\btheta)$ so that
\begin{equation}\label{eq:blockFIM}
    \FIM(\btheta) = \begin{bmatrix}
    \FIM_F(\btheta) & \FIM_{F,N}^\top(\btheta) \\
    \FIM_{F,N}(\btheta) & \FIM_{N}(\btheta) 
\end{bmatrix} \in \mathbb{R}^{(r+P-1) \times (r+P-1)}\,,
\end{equation}
with
\begin{align*}
\FIM_F(\btheta)  &= 
\lambda\begin{bmatrix}
\lambda_{1} \diag(\ba_1)^{-1} &\lambda_{12}^{-1}\bone_{N_1}\bone_{N_2-1}^\top &\dots& \lambda_{1P}^{-1}\bone_{N_1}\bone_{N_p-1}^\top \\
\lambda_{12}^{-1}\bone_{N_2-1}\bone_{N_1}^\top &\lambda_{2}\diag(\ba_{2\bullet})^{-1} & \dots& \lambda_{2P}^{-1}\bone_{N_2-1}\bone_{N_p-1}^\top \\
\vdots & \vdots & \ddots & \vdots\\
\lambda_{1P}^{-1}\bone_{N_p-1}\bone_{N_1}^\top & \lambda_{P2}^{-1}\bone_{N_p-1}\bone_{N_2-1}^\top & \dots & \lambda_{P}\diag(\ba_{P\bullet})^{-1}
\end{bmatrix} \in\mbbR^{r\times r}\,,\\
\FIM_{F,N}(\btheta) & = \lambda\begin{bmatrix}
    \lambda_{21}^{-1}\bone_{N_1}^\top & \bzero_{N_2-1}^\top & \lambda_{23}^{-1}\bone_{N_3-1}^\top &\dots & \lambda_{2P}^{-1}\bone_{N_P-1}^\top \\
    \lambda_{31}^{-1}\bone_{N_1}^\top &\lambda_{32}^{-1}\bone_{N_2-1}^\top & \bzero_{N_3-1}^\top
    &\dots & \lambda_{3P}^{-1}\bone_{N_P-1}^\top \\
    \vdots & \vdots & \vdots & \ddots & \vdots \\
    \lambda_{P1}^{-1}\bone_{N_1}^\top & \lambda_{P2}^{-1}\bone_{N_2-1}^\top & \lambda_{P3}^{-1}\bone_{N_3-1}^\top &\dots & \bzero_{N_P-1}^\top
\end{bmatrix} \in\mbbR^{(P-1)\times r}\,, \\
\intertext{and}
\FIM_N(\btheta) & = \lambda\begin{bmatrix}
    \lambda_2^{-1}a_{21}^{-1} & \lambda_{23}^{-1}&\dots & \lambda_{2P}^{-1}\\ 
    \lambda_{23}^{-1}&  \lambda_3^{-1}a_{31}^{-1}&\dots & \lambda_{3P}^{-1}\\
    \vdots & \vdots & \ddots & \vdots \\
    \lambda_{2P}^{-1} & \lambda_{3P}^{-1} & \dots & \lambda_P^{-1}a_{P1}^{-1}
\end{bmatrix} \in\mbbR^{(P-1)\times (P-1)}\,.
\end{align*}
 Let $
\mathbf{J}_1=\lambda_1\diag(\ba_1) + \sum_{p=2}^P\left(\frac{\lambda_p}{a_{p1}} - 1\right)\ba_1\ba_1^\top ,
$ and 
$\mathbf{J}_p = \lambda_p\diag(\ba_{p\bullet}) + \frac{\lambda_p}{a_{p1}}\ba_{p\bullet}\ba_{p\bullet}^\top$
for $p\in\{2,\dots,P\}$. Then the principal submatrix $\FIM_F(\btheta)$ of the block matrix in Equation \eqref{eq:blockFIM} is nonsingular and repeated applications of the Sherman-Morrison formula imply that
$$
\FIM_F^{-1}(\btheta) = 
\lambda^{-1}\begin{bmatrix}
    \mathbf{J}_1
    &
    -\frac{\lambda_2}{a_{21}}\ba_1\ba_{2\bullet}^\top
    & \dots &
    -\frac{\lambda_P}{a_{P1}}\ba_1\ba_{P\bullet}^\top    
    \\
    -\frac{\lambda_2}{a_{21}}\ba_{2\bullet}\ba_1^\top 
    &
    \mathbf{J}_2
    & \dots &
    0
    \\
    \vdots & \vdots &\ddots & \vdots \\
    -\frac{\lambda_P}{a_{P1}}\ba_{P\bullet}\ba_1^\top   
    & 
    0
    &\dots & 
    \mathbf{J}_P
\end{bmatrix},
$$
and hence 
$\text{rank}\big(\FIM_F(\btheta)\big) = r$.
Furthermore, the Schur complement 
\begin{align} \label{schur-comp}
\FIM(\btheta) /\ \FIM_F(\btheta)  \coloneqq \FIM_N(\btheta) - \FIM_{F,N}(\btheta) \FIM_F^{-1}(\btheta) \FIM_{F,N}^\top(\btheta)(\btheta) = \mathbf{0}
\end{align}
which implies 
$\text{rank}(\FIM_F(\btheta) /\ \FIM_F(\btheta) ) =0$. Hence, the result follows from the Guttman rank additivity formula
$$
\text{rank}\big(\FIM(\btheta)\big) = \text{rank}\big(\FIM_F(\btheta)\big)  + \text{rank}\big(\FIM(\btheta) /\ \FIM_F(\btheta) \big) =r.
$$
Finally, 
$\FIM(\btheta) \mathbf{H} = \mathbf{0}$,
is a consequence of Equation \eqref{schur-comp} and established our third claim.

\end{proof}

We have shown that$\FIM_F(\btheta)$ is a principal submatrix of $\FIM(\btheta)$ and that $\FIM_F(\btheta)$ is nonsingular and has the same rank as $\FIM(\btheta)$. This means that $\btheta_{\bullet}$ suffices to characterize the PCP model. Note that $\FIM_F(\btheta)$ is the Fisher information that corresponds to the smaller parameter vector $\btheta_{\bullet}$.

When constructing $\btheta_{\bullet}$ we chose  $\ba_p^\top =\begin{bmatrix}
a_{p1} & \ba_{p\bullet}^\top
\end{bmatrix}^\top$, meaning that $\ba_{p\bullet}$ is $\ba_{p}$ with its first entry $a_{p1}$ removed. Removing the first entry of $\ba_p$ is not necessary, and we could have chosen to remove any one entry without loss of generality. This is because for $\mat{X}\sim \Poi (\ba_1\ba_2^\top)$, permuting the rows of $\mat{X}$ results in applying the same permutation to the vector $\ba_1$.  In a similar fashion, we could permute the slices of $\ten{X}$ across the $p$th mode, leading to applying the same permutation to the entries of $\ba_p$, and Theorem \ref{thm:Rank1_identf} would still hold. Similarly we chose $\btheta_{\bullet} = \begin{bmatrix} \ba_1^\top &\ba_{2\bullet}^\top & \cdots& \ba_{P\bullet}^\top\end{bmatrix}^\top$, meaning that we chose to remove an entry to all factors but the first one $\ba_1$. Choosing the first one as the one without a removed entry is not necessary, and we could have chosen any one factor vector $\ba_p$ instead, without loss of generality. This is because for $\mat{X}\sim \Poi (\ba_1\ba_2^\top)$, the transpose $\mat{X}^\top\sim \Poi (\ba_2\ba_1^\top)$. Similarly, we could permute the $P$ modes of the tensor $\ten{X}$ and this would result in applying the same permutation to the $P$ factors $\ba_1,\dots,\ba_P$, and Theorem \ref{thm:Rank1_identf} would still hold.

We can also conclude that
\begin{align*}
    \mathbf{P} & = \mathbf{I} - \mathbf{H} \big( \mathbf{H}^\top \mathbf{H}\big)^{-1} \mathbf{H}^\top
\end{align*}
is the orthogonal projector onto the range of $\FIM(\btheta)$ and 
is easily constructed since $\mathbf{H}^\top \mathbf{H}$ is a symmetric positive definite matrix of order $P-1$. Hence, given an estimate $\btheta$, the matrix vector product 
\begin{align*}
    \mathbf{P} \btheta = \btheta - \mathbf{H} \mathbf{s} \text{ where } \mathbf{H}^\top \mathbf{H} \mathbf{s} =  \mathbf{H}^\top \btheta \in \mathbb{R}^{P-1}
\end{align*} 
results in a vector that lies in the range of $\FIM(\btheta)$, which no longer contains components in the direction of the $P-1$ redundant parameters. Indeed, the vector $\mathbf{P} \btheta$ is orthogonal to the nullspace of $\FIM(\btheta)$ of dimension $P-1$. 
The use of $\mathbf{P}$ avoids the possibly arduous task of partitioning $\FIM(\btheta)$ and instead works directly with $\btheta$.

\section{Numerical Experiments}\label{sec:simus}
We present numerical experiments that corroborate our expression for the Fisher information matrix of Theorem \ref{thm:OFIM_genR}, and its conjectured rank in Conjecture \ref{conjecture:rank}. 
In these experiments we consider PCP-distributed random tensors $\ten{X}\sim\Poi(\ten{M})$, where $\ten{M} = [\![\mat{A}_1,\dots,\mat{A}_P]\!]\in\mbbR_{+}^{N\times \dots\times N}$ is a tensor of order $P$ with rank $R$, and its $N^P$ entries average exactly $S$.  We will study different combinations of $N,S,R,P$.

We generate the factor matrices $\mat{A}_1,\dots,\mat{A}_P$ as follows. First, to ensure that the matrix factors are equally weighted, i.e, $\mat{A}_1^\top\bone_{N} = \dots=\mat{A}_P^\top\bone_{N} = \bone_R$, we generate the columns of each $\mat{A}_p\in\mbbR_{+}^{N\times R}$  uniformly at random from a unit simplex. Because the Poisson random variable is degenerate when the Poisson coefficient is zero, we also constrain the simplex to have minimum entry $1/(100N)$. Controlling for $S$ is important because Poisson random variables have equal mean and variance, and hence $S$ adjusts the signal-to-noise ratio. To ensure that the entries of $\ten{M}$ average exactly $S$, a weight vector was chosen as $\blambda \coloneqq (SN^P/\sum_{r=1}^Rr)*\begin{bmatrix}
    1 & \dots &R
\end{bmatrix}^\top$ before reweighting each factor matrix as $\mat{A}_p\leftarrow \mat{A}_p\diag(\blambda^{*1/P})$. Our choice of $\blambda$ ensures that the $R$ different rank one components that make up $\ten{M}$ are weighted differently.

\subsection{Monte Carlo validation of the expected Fisher information matrix}\label{sec:MCFIM}

We will support Theorem \ref{thm:OFIM_genR} by comparing $\FIM(\btheta)$ with a Monte Carlo approximation that uses Bartlett's identity of Equation \eqref{eq:bartletts}.  Our Monte Carlo procedure is as follows. We generate $K$ independent draws $\ten{X}_1,\dots,\ten{X}_K$ from $\Poi(\ten{M})$, and use Equation \eqref{eq:score1} to compute the empirical score values $\bs_1,\dots,\bs_K$.
According to the strong law of large numbers and Bartlett's identity, we have that as $K\rightarrow \infty$,
\begin{equation}\label{eq:MC_FIM}
\bmu_K\coloneqq \dfrac{1}{K}\sum_{k=1}^K \bs_k \overset{a.s.}{\rightarrow} \bzero
,\quad \text{and} \quad 
\widehat\FIM_K(\btheta) \coloneqq \dfrac{1}{K}\sum_{k=1}^K (\bs_k-\bmu_K)(\bs_k-\bmu_K)^\top \overset{a.s.}{\rightarrow} \FIM(\btheta).
\end{equation}
Hence, $\widehat\FIM_K(\btheta)$ is a Monte Carlo approximation to $\FIM(\btheta)$.
For each $\btheta$ generated as in the beginning of Section \ref{sec:simus}, we computed $\FIM(\btheta)$ by Theorem \ref{thm:OFIM_genR} and its approximation $\widehat\FIM_K(\btheta)$ through Equation \eqref{eq:MC_FIM}, where $K=4,16,64,256,1024$, $N=10,25,50,100$, $S=0.1,1,10,100$, $R=1,2,3,4$ and $P=3$.
To quantify how accurately $\widehat\FIM_K(\btheta)$ approximates $\FIM(\btheta)$, we use the relative error
$
||\FIM(\btheta) - \widehat\FIM_K(\btheta)||_F/||\FIM(\btheta)||_F,
$ and to account for sampling variability, for each $\btheta$ we obtained 100 different $\widehat\FIM_K(\btheta)$ and their relative errors. 
\begin{figure}
\centering
\includegraphics[width=0.8\textwidth]{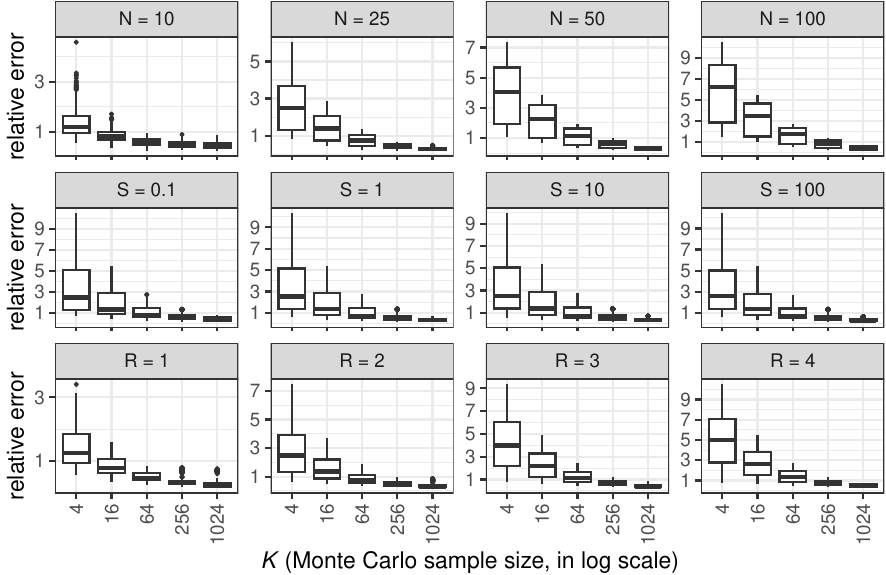}
\caption{Boxplots of relative error $
||\FIM(\btheta) - \widehat\FIM_K(\btheta)||_F/||\FIM(\btheta)||_F
$  between the Monte-Carlo estimated $\widehat\FIM_K(\btheta)$ and our analytical expression for $\FIM(\btheta)$ in Theorem \ref{thm:OFIM_genR}, across different values of $N,S,R,K$.}
\label{fig:MCFIM}
\end{figure}

In Figure \ref{fig:MCFIM} we show boxplots of relative errors of $\widehat\FIM_K(\btheta)$ for different values of $N,S,R$ and $K$, across 100 repetitions. We observe a general decrease in relative error as the Monte Carlo sample size $K$ increases, across all values of $N,S,R$. The maximum relative error (which occurs when $K=4$) is about $\sqrt{N}$. There is also an increase in relative error with increasing $R$, with the $R=1$ case having smaller error rates than the $R=2,3,4$ cases, which have similar error rates. This can be explained through the missing information principle of Equation \eqref{eq:missinginfoprinc}, which states that the  $R=2,3,4$ cases (which have missing information) have an the extra term $\FIM_{m}(\btheta,\bx)$, which is zero for the $R=1$ case. Finally, there is minimal difference in relative error for different values of $S$, indicating that the magnitude of the Poisson rates in $\ten{M}$ don't have an effect in the accuracy of the Monte-Carlo approximation. Hence, this experiment corroborates that our analytical expression for $\FIM(\btheta)$ in Theorem \ref{thm:OFIM_genR} matches Bartlett's identity in Equation \eqref{eq:bartletts}. 

\subsection{Numerical validation of the conjectured Fisher information matrix rank}\label{sec:sim_identif}
We will provide evidence for Conjecture \ref{conjecture:rank} by comparing the conjectured rank of the Fisher information matrix $\FIM(\btheta)$ with its numerical approximation.  For PCP models described at the beginning of Section \ref{sec:simus}, Conjecture \ref{conjecture:rank} states that 
\begin{equation}\label{eq:matrank_equalN}
\text{rank}(\FIM(\btheta)) = \min\Big(PNR-L, N^P\Big),
\end{equation}
where $Q=(\min\{R,N\})^2$ if $P=2$, and $Q=R(P-1)$ if $P>2$.
We generated Fisher information matrices $\FIM(\btheta)$ for the cases where $R=1,2,3,4,5$, $P=2,3,4$, $N=10,25,50,100$, and $S=4$. To account for sampling variability we generated $100$ different $\FIM(\btheta)$ for each combination of $R,P,N$. (explain S does not affect here)

For each Fisher information matrix $\FIM(\btheta)$, we calculated its numerical rank as the number of eigenvalues larger than a small threshold, which we set to the largest eigenvalue times the square root of machine epsilon ($2^{-52}$).
The eigenvalues were numerically approximated using the python function \texttt{numpy.linalg.eigvalsh} \citep{harrisetal20}, which calls LAPACK's \texttt{ssyevd} \citep{andersonetal99} to obtain the eigenvalues through a divide-and-conquer method. In Figure~\ref{fig:FIMrank}(a) we display the ratio between the conjectured and numerical ranks, and display them across all values of $P$, $N$ and $R$. The ratio is one across all combinations, providing numerical evidence for our conjecture.

Conjecture \ref{conjecture:rank} also applies in the underdetermined case where $N^P<PNR-L$. To study this setting, we chose $P = 3$, $N = 8$, $R=1,2,\dots,46$ so that conjecture \ref{conjecture:rank} states that
$$
\text{rank}(\FIM(\btheta))=\begin{cases}
    PNR-L  & R<512/22\approx 23.3 \\
    N^P & R>512/22
\end{cases}.
$$ 
In Figure~\ref{fig:FIMrank}(b) we display two curves as functions of $R$: in green is the ratio of the numerical rank with $N^P$, and in orange is the ratio of the numerical rank with $PNR-L$. We also included a vertical dashed line at $R = 512/22\approx 23.3$. As conjectured, we see a flat orange line in one when $R<512/22$, and a flat green line in one when $R>512/22$. We note that as $R$ increases beyond $R = 23$, the Fisher information matrix $\FIM(\btheta)$ gets larger in dimension, but its rank doesn't go beyond $N^P = 512$.

\begin{figure}
\centering
\begin{tabular}{cc}
\raisebox{.55cm}{
\includegraphics[width=.5\linewidth,page = 1]{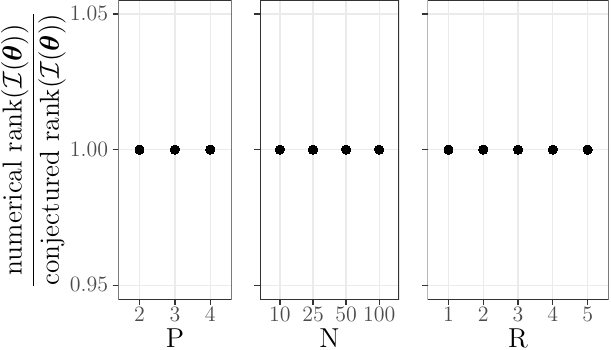}}
&
\includegraphics[width=.35\linewidth,page = 2]{figures/FIMrank-crop.pdf}\\
(a) & (b)
\end{tabular}
\caption{(a) Comparison of numerical and conjectured Fisher information matrix ranks across various values of $P$, $N$ and $R$. The numerical and conjectured ranks are the same in all cases. (b) Comparison of the numerical Fisher information matrix rank against $N^P$ (in green) and $PNR-L$ (in orange) when $P=3$ and $N=8$ are fixed, but the CP rank increases $R=1,2,\dots,46$ . According to our conjecture, we should expect a flat orange line in one when $R<23$ and a flat green line in one when $R>23$.  }
\label{fig:FIMrank}
\end{figure}

\section{Discussion}

In this article we advanced the understanding of the PCP tensor model by introducing a latent-variable formulation that simplified parameter inference, which was traditionally complicated by the presence of sums within log terms of the loglikelihood function. Through this latent variable formulation, we formulated EM algorithms for maximum loglikelihood and framed existing algorithms as special cases. We also derived the observed and expected Fisher information matrices through the use of Oakes’ theorem, which provided crucial insights into the well-posedness of the model, such as the role that CP rank plays in the identifiability and underdeterminacy of the model. Our numerical experiments further validated these theoretical findings, illustrating the effect that an increasing tensor rank has on the rank of the Fisher information matrix, and showing that our application of Oakes’ theorem matches what we would expect from using Bartlett’s identity. Overall, our work enhances the analytical capabilities of tensor models by bridging the fields of latent-variable models and tensor decompositions, with implications for fields that rely on multi-dimensional count data, such as network analysis, geospatial studies, and natural language processing.

There are numerous promising avenues for further investigation. For instance, we can use the tools derived in this article to reformulate other Poisson tensor models, such as the Tucker or Tensor Train models \citep{koldaandbader06}, into latent-variable frameworks. Additionally, we can extend our approach to formulate tensor-response regression or tensor completion as latent-variable models and explore how sample size and number of observed entries influence the Fisher information matrix in these contexts \citep{llosaandmaitra22,lock17,Gandy_2011}. We can also expand our focus beyond Poisson distributions to include other distributions where the sums of random variables belong to the same family \citep{llosaandmaitra24}. Furthermore, we can investigate additional applications of the Fisher information matrix, such as rank selection and stationary point diagnostics.

\section{Appendix}
\appendix
\section{Preliminaries}

Consider the matrix functions of matrix argument $f(\mat{X})\in\mbbR^{k\times l}$ and $g(\mat{X})\in\mbbR^{m\times n}$. Suppose $l=m$ so that the matrix product $f(\mat{X})g(\mat{X})$ can be performed. The following matrix product rule follows from the property $\vecc(\mat{A}\mat{X}\mat{B}^\top) = (\mat{B}\otimes \mat{A})\vecc(\mat{X})$ and the differential $\partial[f(\mat{X})g(\mat{X})] = \partial[f(\mat{X})]g(\mat{X}) + f(\mat{X})\partial[g(\mat{X})]$ \citep{minka97}
\begin{equation}\label{eq:matrix-productrule}
\dfrac{\partial \vecc[f(\mat{X})g(\mat{X})]}{\partial \vecc(\mat{X})}
=
[g(\mat{X})^\top\otimes \mat{I}_k]\dfrac{\partial \vecc[f(\mat{X})]}{\partial \vecc(\mat{X})}
+
[\mat{I}_n\otimes f(\mat{X})]\dfrac{\partial \vecc[g(\mat{X})]}{\partial \vecc(\mat{X})}.
\end{equation}

Now suppose $k=m$ and $l=n$ so that the Hadamard product $f(\mat{X})*g(\mat{X})$ can be performed. Similar to before, the following Hadamard-product rule follows from $\vecc(\mat{A}*\mat{X}) = \dv(\mat{A})\vecc(\mat{X})$ and the differential $\partial[f(\mat{X})*g(\mat{X})] = \partial[f(\mat{X})]*g(\mat{X}) + f(\mat{X})*\partial[g(\mat{X})]$ \citep{minka97}
\begin{equation}\label{eq:hadamard-productrule}
\dfrac{\partial \vecc[f(\mat{X})*g(\mat{X})]}{\partial \vecc(\mat{X})}
=
\dv[g(\mat{X})]\dfrac{\partial \vecc[f(\mat{X})]}{\partial \vecc(\mat{X})}
+
\dv[f(\mat{X})]\dfrac{\partial \vecc[g(\mat{X})]}{\partial \vecc(\mat{X})}.
\end{equation}

Now consider vectors $\bx,\by$ and matrix $\mat{S}$ of sizes such that $\by^\top \mat{S}\bx$ can be performed. The following Jacobian holds
\begin{equation}\label{eq:jac_inv}
\dfrac{\partial \left[\by\oslash(\mat{S}\bx)\right]}{\partial \bx}
= -\diag\left[\by\oslash(\mat{S}\bx)^{*2}\right]\mat{S}.
\end{equation}

The following property of Kathri-Rao matrix products holds for any $k=1,2,\dots,P$
\begin{equation}\label{eq:vectorKR}
\vecc(\odot\mat{A}_{[P]}) = \mat{B}_{-k} \vecc(\mat{A}_k),
\end{equation}
where $\mat{B}_{-k}$ is a $R\times R$ block-diagonal matrix  with $(r,r)
$ block $\bigotimes_{p=P}^{1} \mat{B}_{pkr}$, where 
$$
\mat{B}_{pkr}=\begin{cases}
    \mat{I}_{N_k} & p=k\\
    \mat{A}_p(:,r) & p\neq k
\end{cases}.
$$
For example, when $P=2$ we have
$$
\mat{B}_{-1} = \begin{bmatrix}
    \mat{A}_2(:,1) \otimes \mat{I}_{N_1} &&\\
    &\ddots&\\
    &&\mat{A}_2(:,R) \otimes \mat{I}_{N_1}     
\end{bmatrix}
,\quad
\mat{B}_{-2} = \begin{bmatrix}
      \mat{I}_{N_2} \otimes\mat{A}_1(:,1)&&\\
    &\ddots&\\
    && \mat{I}_{N_2} \otimes  \mat{A}_1(:,R)
\end{bmatrix}.
$$

\section{Derivatives of \texorpdfstring{$Q$}{Q}}\label{app:derQ1}

In Section \ref{sec:cpaprLatent} we derived the expected complete loglikelihood $Q(\btheta,\bar \btheta)$. Its gradients and Jacobians are important as per Theorem \ref{thm:oakesHess}.

\begin{lemma}\label{lemma:gradQ1}
For a fixed $\bar\btheta$, and with respect to the parameter vector $\btheta = [\vecc(\mat{A}_1)^\top\dots\vecc(\mat{A}_P)^\top]^\top$, the gradient of $Q(\btheta,\bar\btheta)$ in Equation \eqref{eq:Qfunc2} can be written as
\begin{equation}\label{eq:proof_gradQ1}
\dfrac{\partial}{\partial\btheta} Q(\btheta,\bar\btheta)
=
 \begin{bmatrix}\dv(\mat{A}_1^{*-1}) \vecc(\bar{\mat{Z}}_1) -(*\blambda_{[P]\setminus\{1\}}\otimes \bone_{N_1}) 
\\\vdots\\
 \dv(\mat{A}_P^{*-1}) \vecc(\bar{\mat{Z}}_P) - (*\blambda_{[P]\setminus\{P\}}\otimes \bone_{N_P}) 
 \end{bmatrix},
\end{equation}
and the Hessian is a $P\times P$ block matrix 
\begin{equation*}\label{eq:proof_hessQ1}
\dfrac{\partial^2}{\partial\btheta\,\partial\btheta^\top} Q(\btheta,\bar\btheta)
=
\left\{
\bar{\mat{G}}_{k,l}
\right\}_{k,l},
\end{equation*}
with $(k,l)$ sub-block
\begin{equation*}\label{eq:proof_hessQ1_subblock}
\bar{\mat{G}}_{k,l} = 
-\begin{cases}
\dv(\bar{\mat{Z}}_k\oslash\mat{A}_k^{*2})
& k = l \\
\diag(*\blambda_{[P]\setminus\{k,l\}})\otimes (\bone_{N_k}\bone_{N_l}^\top)
&k\neq l
\end{cases}.
\end{equation*}
When $P=2$, the term $\diag(*\blambda_{[P]\setminus\{k,l\}})$ above is replaced with the identity matrix $\mat{I}_R$.
\end{lemma}
\begin{proof}
Equation \eqref{eq:proof_gradQ1} is a block-matrix with $k$th vertical block
\begin{equation*}
    \begin{aligned}
    \dfrac{\partial Q(\btheta,\bar\btheta)}{\partial \vecc (\mat{A}_k)} 
    =
    \vecc\Big( \dfrac{\partial Q(\btheta,\bar\btheta)}{\partial \mat{A}_k}\Big)
    &= 
    \vecc\Big( 
    \bar{\mat{Z}}_k * \mat{A}_k^{*-1}\Big) - \dfrac{\partial }{\partial \vecc (\mat{A}_k)} \bone_{N_k}^\top\mat{A}_p(*\blambda_{[P]\setminus\{k\}}))
    \\&= 
    \dv(\mat{A}_k^{*-1}) \vecc(\bar{\mat{Z}}_k) -(*\blambda_{[P]\setminus\{k\}}\otimes \bone_{N_k}).
    \end{aligned}
\end{equation*}
From the above we obtain
\begin{equation*}
    \begin{aligned}
    \bar{\mat{G}}_{k,k} := \dfrac{\partial Q(\btheta,\bar\btheta)}{(\partial \vecc \mat{A}_k)(\partial \vecc \mat{A}_k)^\top} 
    &= 
    \dfrac{\partial}{(\partial \vecc \mat{A}_k)^\top} \vecc\Big( 
    \bar{\mat{Z}}_k * \mat{A}_k^{*-1}
    \Big)
    \\&= 
    -\dv( \bar{\mat{Z}}_k * \mat{A}_k^{*-2}) .
    \end{aligned}
\end{equation*}
When $P=2$ we have
\begin{equation*}
    \begin{aligned}
    \bar{\mat{G}}_{k,l} := \dfrac{\partial Q(\btheta,\bar\btheta)}{(\partial \vecc \mat{A}_k)(\partial \vecc \mat{A}_l)^\top} 
    &= 
    -\dfrac{\partial}{(\partial \vecc \mat{A}_l)^\top} \Big( 
    *\blambda_{[P]\setminus\{k\}}\otimes \bone_{N_k}
    \Big).
    \\&=
    -\dfrac{\partial}{(\partial \vecc \mat{A}_l)^\top} \Big( 
    \left[\diag(*\blambda_{[P]\setminus\{k,l\}})\otimes (\bone_{N_k}\bone_{N_l}^\top)\right]\vecc(\mat{A}_l)
    \Big)
    \\&=
    -\diag(*\blambda_{[P]\setminus\{k,l\}})\otimes (\bone_{N_k}\bone_{N_l}^\top),
    \end{aligned}
\end{equation*}
and the $P=2$ case follows similarly but with $*\blambda_{[P]\setminus\{k\}}\otimes \bone_{N_k}
=
\left[\mat{I}_R\otimes (\bone_{N_k}\bone_{N_l}^\top)\right]\vecc(\mat{A}_l).$

\end{proof}

\begin{lemma}\label{lemma:Zders}
Let $\bar{\mat{Z}}_k = \bar{\mat{A}}_k * [\bar{ \mat{E}}_k (\odot\bar{\mat{A}}_{[P]\setminus\{k\}})]$, $\bar{ \mat{E}}_k = \mat{X}_{(k)}\oslash (\bar{\mat{A}}_k(\odot\bar{\mat{A}}_{[P]\setminus\{k\}})^\top) $ from Equation \eqref{eq:Qfunc2}, using the notation of Theorem \ref{thm:OFIM_genR}.  Also, define $\bar{\mat{B}}_{-k,l}$ as the matrix that satisfies
\begin{equation*}\label{eq:Anokl}
\vecc(\odot\bar{\mat{A}}_{[P]\setminus\{k\}}) = \bar{\mat{B}}_{-k,l} \vecc (\bar{\mat{A}}_l),
\end{equation*}
 which can be constructed according to Equation \eqref{eq:vectorKR}. 
  Then
\begin{equation*}
\dfrac{\partial^2}{\partial\btheta\,\partial\bar\btheta^\top} Q(\btheta,\bar\btheta)
=\left\{
\dv(\bar{\mat{A}}_k\oslash \mat{A}_k)\bar{\mat{H}}_{k,l}
\right\}_{k,l}.
\end{equation*}
When $k=l$ we have 	
$$
\bar{\mat{H}}_{k,k}=
\dv\left[\bar{\mat{Z}}_k\oslash \bar{\mat{A}}^{*2} \right] -
\big\{ 
\diag\left[
(\ten{X} \oslash \ten{\bar M}^{*2}) \bar\times_{q\neq k}(\bar\ba_{q,r}*\bar\ba_{q,s})
\right]
\big\}_{r,s},
$$
and when $k\neq l$,
$$
\bar{\mat{H}}_{k,l} =
\begin{cases}
(\mat{I}_R\otimes \bar{ \mat{E}}_k)
-\left\{ 
[\bar\ba_{k,s}\bar\ba_{l,r}^\top]*
(\mat{X} \oslash \bar{\mat{M}}^{*2})_{(k)}
\right\}_{r,s}
 & \text{if } P =2 \\
(\mat{I}_R\otimes \bar{ \mat{E}}_k)\bar{\mat{B}}_{-k,l}
-
\left\{ 
[\bar\ba_{k,s}\bar\ba_{l,r}^\top ]*
\left[
(\ten{X} \oslash \ten{\bar M}^{*2})\bar\times_{q\neq k,l} 
(\bar\ba_{q,r}*\bar\ba_{q,s})
\right]
\right\}_{r,s} &\text{if } P >2
\end{cases}.
$$
\end{lemma}
\begin{proof}
Differentiating the gradient of Equation \eqref{eq:proof_gradQ1} with respect to $\bar\btheta$ we obtain 
\begin{equation}\label{eq:Djac1}
\dfrac{\partial^2}{\partial\btheta\,\partial\bar\btheta^\top} Q(\btheta,\bar\btheta)
=\left\{
\dv(\mat{A}_k^{*-1})\dfrac{\partial \vecc(\bar{\mat{Z}}_k)}{\partial\vecc(\bar{\mat{A}}_l)}
\right\}_{k,l}.
\end{equation}
Equation \eqref{eq:Djac1} involves only derivatives of $\bar{\mat{Z}}_k$, which we can write in terms of derivatives of $\bar{ \mat{E}}_k$ only using the product rules we derived in Equations \eqref{eq:matrix-productrule} and \eqref{eq:hadamard-productrule}
\begin{equation}\label{eq:dZk1}
\dfrac{\partial \vecc(\bar{\mat{Z}}_k)}{\partial\vecc(\bar{\mat{A}}_l)}
=
\begin{cases}
\dv(\bar{\mat{A}}_k)\left[
\dv(\bar{\mat{Z}}_{k}\oslash \bar{\mat{A}}_k^{*2}) + 
((\odot\bar{\mat{A}}_{[P]\setminus\{k\}})^\top\otimes \mat{I}_{N_k}) 
\dfrac{\partial \vecc(\bar{ \mat{E}}_k)}{\partial\vecc(\bar{\mat{A}}_k)}
\right] & k= l \\
\dv(\bar{\mat{A}}_k)\left[
(\mat{I}_R\otimes \bar{ \mat{E}}_k)\bar{\mat{B}}_{-k,l}
+
((\odot\bar{\mat{A}}_{[P]\setminus\{k\}})^\top\otimes \mat{I}_{N_k})\dfrac{\partial \vecc(\bar{ \mat{E}}_k)}{\partial\vecc(\bar{\mat{A}}_l)}
\right]
 & k\neq l
\end{cases}.
\end{equation}
When $P=2$, $\odot\bar{\mat{A}}_{[P]\setminus\{k\}}=\mat{A}_{l}$, and so  $\bar{\mat{B}}_{-k,l}$ is an identity matrix. 
To find the derivatives of $\bar{\mat{E}}_k$ first note that $\vecc(\bar{\mat{E}}_k) = K_{(k)} \vecc(\ten{X}\oslash\bar{\ten{M}})$, where $K_{(k)}$ is a permutation matrix defined in \citep[Lemma 2.1f]{llosaandmaitra22} that generalizes the commutation matrix \citep{magnus79}. Hence,
\begin{equation}\label{eq:dEk}
\begin{aligned}
\dfrac{\partial \vecc(\bar{ \mat{E}}_k)}{\partial\vecc(\bar{\mat{A}}_l)} 
& = 
K_{(k)}\dfrac{\partial }{\partial\vecc(\bar{\mat{A}}_l)}
\vecc(\ten{X}\oslash\bar{\ten{M}})
\\&= 
K_{(k)}\dfrac{\partial }{\partial\vecc(\bar{\mat{A}}_l)}
\Big[\vecc(\ten{X})\oslash\vecc(\bar{\ten{M}})\Big]
\\&= 
K_{(k)}\dfrac{\partial }{\partial\vecc(\bar{\mat{A}}_l)}
\Big[\vecc(\ten{X})\oslash(\bar{\mat{B}}_{-l}\vecc(\bar{\mat{A}}_l))\Big]
\\&= 
-K_{(k)}\dv(\ten{X}\oslash\bar{\ten{M}}^{*2})\bar{\mat{B}}_{-l},
\end{aligned}
\end{equation}
where the last equality follows from the identity in Equation \eqref{eq:jac_inv}. 
The remainder of the proof follows from plugging Equation \eqref{eq:dEk} into \eqref{eq:dZk1}, and \eqref{eq:dZk1} into \eqref{eq:Djac1}, and simplifying. 
\end{proof}

\section{Proof of Theorem \ref{thm:OFIM_genR}}\label{proof:Hessian}

\begin{proof}
We obtain the loglikelihood Hessian as a result of Oakes' Theorem (Theorem \ref{thm:oakesHess}):
$$
\dfrac{\partial^2}{\partial \btheta\, \partial \btheta^\top} 
\ell(\btheta)
= 
\left[
\dfrac{\partial^2}{\partial \btheta\, \partial \btheta^\top} Q(\btheta,\bar\btheta)
+
\dfrac{\partial^2}{\partial \btheta\, \partial \bar\btheta^\top} Q(\btheta,\bar\btheta)
\right]_{\bar\btheta = \btheta}.
$$
Plugging our expressions for $\frac{\partial^2}{\partial \btheta\, \partial \btheta^\top} Q(\btheta,\bar\btheta)$ from Lemma \ref{lemma:gradQ1}, and for $\frac{\partial^2}{\partial \btheta\, \partial \bar\btheta^\top} Q(\btheta,\bar\btheta)$ from Lemma \ref{lemma:Zders} into the above equation, and evaluating  $\bar \btheta = \btheta$, leads to the observed Fisher information
\begin{equation*}\label{eq:hessllhd}
\FIM_{obs}(\btheta,\bx):= -\dfrac{\partial^2}{\partial \btheta\, \partial \btheta^\top} 
\ell(\btheta)
=
\left\{
-(\bar{\mat{G}}_{k,l} + \bar{\mat{H}}_{k,l})_{\bar\btheta = \btheta}
\right\}_{k,l}.
\end{equation*}
 After cancellation we have when $k=l$ that
$$
-(\bar{\mat{G}}_{k,k} + \bar{\mat{H}}_{k,k})_{\bar\btheta = \btheta} = \left\{ 
\diag\left[
(\ten{X} \oslash \ten{M}^{*2})\bar{\times}_{q\neq k} (\ba_{q,r}*\ba_{q,s})
\right)
\right\}_{r,s},
$$
and when $k\neq l$ that
$$
-(\bar{\mat{G}}_{k,l} + \bar{\mat{H}}_{k,l})_{\bar\btheta = \btheta} = 
\left\{ 
\left[\ba_{k,s}\ba_{l,r}^\top \right]*
\left[
(\ten{X} \oslash \ten{M}^{*2})\bar\times_{q\neq k,l}
(\ba_{q,r}*\ba_{q,s})
\right]
\right\}_{r,s}
+
\mat{F}_{k,l},
$$
where 
$\mat{F}_{k,l} = \diag(*\blambda_{[P]\setminus\{k,l\}}))\otimes (\bone_{N_k}\bone_{N_l}^\top) -(\mat{I}_R\otimes \mat{E}_k) \mat{B}_{-k,l}$ is a block-diagonal matrix simplified in the theorem statement. This finds the observed Fisher information.
For the expected Fisher information, note that since $\mbbE(\ten{X})=\ten{M}$ we have $\mbbE(\ten{X} \oslash \ten{M}^{*2})=\ten{M}^{*-1}$ and $\mbbE(\mat{F}_{k,l})=0$. Hence, the expression for $\FIM(\btheta)$ is the same as for $\FIM_{obs}(\btheta,\bx)$, except that $\mat{F}_{k,l}$ is removed, and $\ten{X} \oslash \ten{M}^{*2}$ is replaced with $\ten{M}^{*-1}$.

\end{proof}

\bibliographystyle{plainnat} 
\bibliography{references} 

\end{document}